\documentclass[a4paper,11pt]{amsart}

\usepackage[T1]{fontenc}
\usepackage{amsmath, amssymb, amsthm, mathtools}
\usepackage{xcolor}
\usepackage{booktabs}

\theoremstyle{definition}
\newtheorem{theorem}{Theorem}[section]

\newtheorem{proposition}[theorem]{Proposition}
\newtheorem{proposition-definition}[theorem]{Proposition--Definition}

\newtheorem{corollary}[theorem]{Corollary}
\newtheorem{definition}[theorem]{Definition}

\newtheorem{example}[theorem]{Example}
\newtheorem{question}[theorem]{Question}

\theoremstyle{definition}
\newtheorem{remark}[theorem]{Remark}

\newtheorem{claim}[theorem]{Claim}

\usepackage[pagebackref=true, colorlinks, citecolor=blue]{hyperref}

\usepackage[numbers,sort]{natbib}
\numberwithin{equation}{section}
\title{Spectrum of the refined Diophantine exponent}
\author{Quang-Khai Nguyen}
\address{Institut Camille Jordan, Universit\'e Claude Bernard Lyon 1 \newline \indent 21 avenue Claude Bernard, 69100 Villeurbanne, France}
\date{\today}
\email{nguyen@math.univ-lyon1.fr}
\makeatletter
\@namedef{subjclassname@2020}{2020 Mathematics Subject Classification}
\makeatother
\begin{document}
\begin{abstract}
The refined Diophantine exponent, recently introduced by the author, is a measure of the periodicity of an infinite word with many applications in transcendental number theory. In this article, we study this exponent from combinatorial viewpoints. First, we show that, over a ternary alphabet, the spectrum of the refined Diophantine exponent is $[1,\infty]$. Second, we establish measure theoretic and topological properties of the refined Diophantine exponent, showing that it equals $1$ almost surely and that each of its level sets is dense in the Cantor topology. Finally, we provide concrete examples with the Champernowne, Rudin--Shapiro, and Thue--Morse words, words coming from coding a rotation by intervals, and bracket words.

\end{abstract}
\maketitle
\tableofcontents
\section{Introduction}


For an infinite word $\mathbf{a}=a_0a_1\cdots$, the \emph{refined Diophantine exponent}, denoted by $\mathbf{Rdio}(\mathbf{a})$, measures how far the word is from being eventually periodic. It is a refinement of the \emph{Diophantine exponent} introduced by Adamczewski and Bugeaud \cite{Adamczewski-Bugeaud-2007-dynamics} that takes a certain notion of \emph{mismatch} (or \emph{noise}) into account, making the exponent more flexible. This notion is motivated by problems from transcendental number theory; we refer the reader to \cite{Khai-2026} and the references therein for more details. Notably, by using the Subspace Theorem of Schmidt \cite{Schmidt-1980,Schmidt-1993}, Theorem~A in \cite{Khai-2026} states that if $\mathbf{a}$ is written over a finite set of algebraic numbers, $\beta$ is an algebraic number such that $|\beta|>1$, and $\mathbf{Rdio}(\mathbf{a})>{M(\beta)}/{\log|\beta|}$, then the number 
\[\sum_{i\geq0}a_i\beta^{-i}\]
either lies in the number field $\overline{\mathbb Q}(\beta,a_i:i\geq0)$ or is transcendental. Here, $M(\beta)$ denotes the Mahler measure of $\beta$. Furthermore, understanding the $\mathbf{Rdio}$ also helps us to understand transcendence measures, see Theorems~B, C and D in \emph{loc. cit.} However, all the examples of words studied in \emph{loc. cit.} have an infinite refined Diophantine exponent. It is therefore natural to ask about the spectrum of this exponent in order to understand to what extent the results of \cite{Khai-2026} can be applied. The spectrum of the Diophantine exponent has been extensively studied in \cite{Dubickas-2009,Adamczewski-Bugeaud-2011,Peltomaki-2024}. We also note the study of another important exponent, the \emph{critical exponent}, whose spectrum has been investigated systematically; see e.g. \cite{Krieger-Shallit-2007,Currie-Rampersad-2008} and references therein. Contrary to those exponents whose spectrum is easy to show to be $[1,\infty]$, the case of the refined Diophantine exponent is more difficult, particularly for values in $(1,\infty)$ due to the allowance of mismatches. Our first main result establishes that the spectrum of the refined Diophantine exponent is as large as possible.

\begin{theorem}\label{Theorem: spectrum of Rdio}
  Over a ternary alphabet, the spectrum of $\mathbf{Rdio}$ is $[1,\infty]$. 
\end{theorem}

To prove Theorem~\ref{Theorem: spectrum of Rdio}, we first find an infinite binary word with very few mismatches (in particular, its $\mathbf{Rdio}$ equals $1$), or, more intuitively, a word that appears \emph{pseudorandom}. Following an idea from \cite[Section~2]{Khai-2026}, we modify this word using lacunary sequences and show that it satisfies the required property. In combinatorics on words, a natural way to measure \emph{pseudorandomness} is through the notion of correlations; see the seminal works \cite{Mauduit-Sarkozy-1997-I,Mauduit-Sarkozy-1998-II,Mauduit-1999-III,Mauduit-1999-IV,Cassaigne-Mauduit-Sarkozy-2002-VII}. Motivated by these studies, we use the probabilistic method to show the existence of such an infinite binary word that satisfies a certain \emph{effective pseudorandomness} property, see Theorem~\ref{theorem: existence of word}. In contrast to the Diophantine exponent, our method requires the alphabet with at least three letters.

Theorem~\ref{Theorem: spectrum of Rdio} explains the strength of the refined Diophantine exponent, since it would mean that \cite[Theorem~A]{Khai-2026} can be applied to words for which results from \cite{Luca-Ouaknine-Worrell-2025,Luca-Ouaknine-Worrell-2023,Kebis-Luca-Ouaknine-Scoones-Worrell-2024,Kebis-Luca-Ouaknine-Scoones-Worrell-2026} cannot be used. 

The second aim of this article is to study topological properties of $\mathbf{Rdio}$. In Diophantine approximation, it is traditional to study almost sure values of exponents from transcendental number theory, such as the irrationality measure and the transcendence measures. Typical results in this direction are topological properties of the set of Liouville numbers (see e.g. \cite[Chapter~2]{Oxtoby-1980}) or Khintchine-type results (see e.g. \cite[Chapters~6 and 8]{Bugeaud-ApproximationbyAlgebraicNumbers}), in which probabilistic methods play an important role. We also refer to \cite{German-2026} and the references therein for recent results on other exponents.  Our next result is motivated by this point of view. 

We use $\mathbb{N}=\{0,1,2,\ldots\}$. For a finite alphabet $\Sigma$ with $|\Sigma|\geq2$, we equip the set $\Sigma^\mathbb N$ of infinite words over $\Sigma$ with the \emph{uniform Bernoulli product measure}; that is, we view each letter as an independent random variable uniformly distributed over $\Sigma$. On the other hand, when $\Sigma$ is not necessarily finite, we can always equip $\Sigma^\mathbb N$ with the \emph{Cantor topology}, i.e., the topology whose basic open sets are the \emph{cylinder sets}, defined as follows: for any finite word $w$ over $\Sigma$, the cylinder set $[w]$ is the set of all infinite words that have $w$ as their prefix. We have:

\begin{theorem}\label{Theorem: probability of Rdio=1}
Let $\Sigma$ be an alphabet. 
\begin{enumerate}
    \item If $\Sigma$ with $|\Sigma|\geq2$ is finite, we equip $\Sigma^\mathbb N$ with the uniform Bernoulli product measure. Then $\mathbb{P}(\mathbf{a}\in\Sigma^\mathbb N:\mathbf{Rdio}(\mathbf{a})=1)=1$.
    \item If $|\Sigma| \geq 3$, we equip  $\Sigma^\mathbb N$ with the Cantor topology. Then, for any $c\in[1,+\infty]$, the set of infinite words $\mathbf{a}$ over $\Sigma$ such that $\mathbf{Rdio}(\mathbf{a})=c$ is dense.
\end{enumerate}
\end{theorem}

The first assertion of Theorem~\ref{Theorem: probability of Rdio=1}, saying that almost surely we have $\mathbf{Rdio}(\mathbf{a})=1$, is an expected phenomenon in combinatorics on words.
Its proof uses standard probabilistic arguments based on the Borel--Cantelli lemma. This approach is classical and has been widely used, for example, in studying the pseudorandomness of words \cite{Alon-Kohayakawa-Mauduit-Moreira-rodl}. Our result implies that for almost all infinite words, we cannot use the strategy based on the Subspace Theorem to obtain the rational-transcendence dichotomy. On the other hand, the second assertion, which is a direct consequence of Theorem~\ref{Theorem: spectrum of Rdio}, says that we can apply \cite[Theorem~A]{Khai-2026} for a dense set of infinite words. In addition, the same proof yields that Theorems~\ref{Theorem: probability of Rdio=1} also holds for the Diophantine exponent.



As in Diophantine approximation, once the spectrum of an exponent is determined using the probabilistic method, it is natural to study concrete examples such as the irrationality measures and transcendence measures of classical transcendental numbers like $e, \pi,$ and $\log 2$. A similar situation arises with $\mathbf{Rdio}$. The previously mentioned examples of infinite words with $\mathbf{Rdio} \in [1, \infty)$ are not explicitly constructed. In the remainder of this article, we study the (refined) Diophantine exponents of specific families of infinite words. Recall that an infinite word is said to be \emph{aperiodic} if it is not eventually periodic. The state of knowledge regarding these exponents is summarized in Table~\ref{table:examples}. 

\begin{table}[htpb]
\hspace*{-2cm}
\renewcommand{\arraystretch}{1.3} 
\begin{tabular}{|l|c|c|}
\hline
\textbf{Word} & \textbf{$\mathbf{Dio}$} & \textbf{$\mathbf{Rdio}$} \\
\hline
Sublinear complexity  & \begin{tabular}{@{}c@{}} $>1$ if aperiodic \\ $\infty$ if eventually  periodic\end{tabular}  &  \\
\hline
Automatic  & \begin{tabular}{@{}c@{}}finite $>1$ if aperiodic \\ $\infty$ if eventually  periodic\end{tabular}  & finite/$\infty$ in some cases  \\
\hline

Lacunary  & finite iff $\limsup{u_{n+1}}/{u_n}<\infty$ & $\infty$ \\
\hline
Sturmian  & \begin{tabular}{@{}c@{}}finite $>2$ if $\theta$ badly approximable \\ $\infty$ if $\theta$ well approximable\end{tabular} & $\infty$ \\
\hline
$k$-bonacci  & finite  & $\infty$ \\
\hline
Champernowne  & $=1$ & $=1$ \\
\hline
Rudin--Shapiro &  $\leq 4$ & $\leq25$ \\
\hline
Thue--Morse  & finite $\geq{3}/{2}$ & $\leq 25$ \\
\hline
\begin{tabular}{@{}c@{}}Coding rotations by intervals  \\ \hspace{-1.2cm}= Bracket of degree $1$\end{tabular} & \begin{tabular}{@{}c@{}}finite if $\theta$ badly approximable \\ $\infty$ if $\theta$ well approximable\end{tabular} & $\infty$ \\
\hline

Bracket of degree $\ell\geq2$ & finite/$\infty$ in some cases & finite/$\infty$ in some cases \\
\hline
\end{tabular}
\vspace{0.2cm}
\caption{The (refined) Diophantine exponents of some families of infinite words}
\label{table:examples}
\end{table}

We remark from the definition (see Section \ref{section: Rdio}) that $\mathbf{Rdio}(\mathbf{a}) \geq \mathbf{Dio}(\mathbf{a})$. In Table~\ref{table:examples}, the results in the first four rows follow from \cite{Dubickas-2009,Adamczewski-Bugeaud-2011,Adamczewski-2010,Peltomaki-2024,Khai-2026}. We will explain the motivation for the remaining entries in the rest of this introduction. The proofs for these results require diverse techniques from combinatorics on words.

\subsection{The Champernowne, Rudin--Shapiro, and Thue--Morse words}

We begin with the \emph{Champernowne}, \emph{Rudin--Shapiro}, and \emph{Thue--Morse} words, motivated by the work of Mauduit and S\`ak\"{o}zy \cite{Mauduit-Sarkozy-1998-II}. These are also classical infinite words that always serve as test cases in combinatorics on words. In particular, the latter two are standard examples of automatic sequences.

For the Champernowne word, the proof is straightforward. The most interesting case is that of the Thue--Morse word. While the finiteness of the Diophantine exponent of the Thue--Morse word follows easily from the fact that it is \emph{overlap-free} (see \cite[Section 8.1.3.3]{Adamczewski-Bugeaud-2010-CANT}), the finiteness of its refined Diophantine exponent is less immediate. The proof is inspired by the careful analysis of correlations in \cite{Mauduit-Sarkozy-1998-II}. More precisely, we adapt their computations to our setting to obtain a bound on the number of mismatches, which yields the desired bound for $\mathbf{Rdio}$. The result for the Rudin--Shapiro word is proved similarly.

\subsection{Coding of an irrational rotation by intervals}
Our next result involves the refined Diophantine exponent of words arising from the \emph{coding of an irrational rotation by intervals}. Typical examples include Sturmian words and words arising from the coding of an irrational rotation by rational intervals, which are known to have an infinite refined Diophantine exponent (see \cite[Section~2]{Khai-2026}). We extend this to the general case by proving that the refined Diophantine exponent of a word arising from the coding of an irrational rotation by \emph{any} intervals is always infinite, unless in the trivial case. The proof of this result uses several arguments regarding continued fractions, which is a standard approach to studying the dynamical properties of this family of words (see e.g. \cite{Berthe-Holton-Zamboni-2006}). 

Our motivation for studying this family of words stems from \emph{degree sequences} in algebraic dynamics, drawing inspiration from \cite{BDJ-2020,Khai-2025,Khai-2026}. Specifically, by examining \emph{monomial surface self-maps} of a \emph{projective toric surface} (associated with integral matrices whose eigenvalues are complex conjugates with arguments \emph{incommensurable with $2\pi$}) and their associated degree sequences with respect to an ample divisor, one observes that for almost all primes $p$, the reduction of such a sequence modulo $p$ can be determined in two steps. First, we evaluate the integral values of a \emph{piecewise-linear function} on $\mathbb{C}$ at $\chi^n$, $n\geq0$, for some Gaussian integer $\chi$ whose argument is incommensurable with $2\pi$; this step provides an irrational rotation. Next, we reduce these values modulo $p$, which results in (not necessarily rational) intervals (see \cite[Section~2]{Khai-2025}), which is similar to the coding of the rotation ${\mathrm{arg}\chi}/{2\pi}\in[0,1)\setminus\mathbb{Q}$. It is worth noting that the proof of Theorem~D in \emph{loc. cit.} shows that under these conditions on the self-maps, this sequence, upon reduction modulo $p$, is not \emph{$p$-automatic}.
\subsection{Bracket words}
Our final object of study is a more general family of words than the coding of an irrational rotation by intervals, namely the \emph{bracket words}. These words are defined by composing a \emph{finitely-valued generalized polynomial} with a \emph{piecewise-constant function} (see e.g. \cite{Adamczewski-Konieczny-2023} and references therein). Motivated by the {pseudorandomness} properties of sequences determined by the fractional part of $n\alpha$, $n^2\alpha$, or more generally $n^k\alpha$ \cite{Mauduit-2000-V, Mauduit-2000-VI}, we expect that bracket words exhibit \emph{pseudorandom} behavior unless they are eventually periodic. In light of this observation, our next result shows that for certain families of bracket words, the refined Diophantine exponents are finite/infinite. The proofs rely on several equidistribution results.

We remark that the reduction modulo $p$ of the degree sequence mentioned previously is not a bracket word. This is because this degree sequence arises from the values of a piecewise-linear function on $\mathbb{C}$, and thus involves the \emph{exponential function}. On the other hand, much like the degree sequence modulo $p$ \cite[Theorem~D]{Khai-2025}, bracket words are not automatic unless they are eventually periodic \cite[Theorem~B]{Konieczny-2022}. We also note that the lacunary sequence defined by $d^n$ for any $d\geq2$ has an infinite refined Diophantine exponent. On the other hand, we have seen that the refined Diophantine exponents of the Thue--Morse and Rudin--Shapiro words (both of which are $2$-automatic) are finite. Meanwhile, the Diophantine exponent of an aperiodic automatic word is always finite, see \cite[Lemma~6.1]{Adamczewski-Cassaigne-2006-Compositio}. Therefore, it is natural to study when the refined Diophantine exponent of an aperiodic automatic word is finite.

\subsection*{Organization}
In Section~\ref{section: Preliminaries}, we recall the definition of the (refined) Diophantine exponent and some of its properties, together with some results on continued fractions. In Section~\ref{section: spectrum of Rdio}, we study the spectrum of the Diophantine exponent. We then investigate its topological properties in Section~\ref{section: topo properties}. Sections~\ref{section: Champernowne} and~\ref{section: Thue--Morse} are devoted to the study of the refined Diophantine exponent of the Champernowne, Rudin--Shapiro, and Thue--Morse words. In Sections~\ref{section: coding rotations} and~\ref{section: bracket}, we study words arising from the coding of rotations by intervals and, more generally, bracket words. Finally, we pose several open questions in Section~\ref{section: conjectures}.

\section{Preliminaries}\label{section: Preliminaries}
\subsection{The refined Diophantine exponent}\label{section: Rdio}

In this section, we present the definition of the refined Diophantine exponent and give several of its basic properties. To motivate this definition, we start with the Diophantine exponent, first introduced in \cite{Adamczewski-Bugeaud-2007-dynamics}.

\begin{definition}\label{Definition: Dio}
The \emph{Diophantine exponent} of a word $\mathbf{a}$, denoted by $\mathbf{Dio}(\mathbf{a})$, is defined as the supremum of the real numbers $\rho$ for which there exist three sequences of integers $(r_n)_{n\geq 0}$, $(s_n)_{n\geq 0}$, and $(t_n)_{n\geq 0}$ satisfying, for all $n\geq 0$:
\begin{itemize}
\item[(i)] $-1\leq r_n<s_n<t_n$;
\item[(ii)] the words $\mathbf{a}[r_n+1, r_n+t_n-s_n]$ and $\mathbf{a}[s_n+1, t_n]$ coincide;
\item[(iii)]  $t_n \geq \rho s_n$;
\item[(iv)]   $s_n - r_n \to \infty$ as $n \to \infty$.
\end{itemize}Here, $\mathbf{a}[r,s]$ denotes the finite word $a_r a_{r+1} \cdots a_s$.
\end{definition}

\begin{remark}\label{remark:Dio}
Originally, this definition can be reformulated in the following more intuitive form: $\mathbf{Dio}(\mathbf{a})$ is defined as the supremum of the real numbers $\rho$ for which there exist two sequences of finite words $U_n$ and $V_n$, and a real number  $w>1$ such that $U_nV_n^w$ is a prefix of $\mathbf{a}$, ${|U_nV_n^w|}/{|U_nV_n|}\geq \rho$, and $|V_n|$ tends to infinity. Here, for a finite word $V$ and a real number $w$, we write $V^{w}$ for the word $V^{\lfloor w\rfloor}V'$ with $V'$ being the prefix of $V$ of length $\lceil\{w\}|V|\rceil$, where $\lfloor\cdot\rfloor$ (resp. $\lceil\cdot\rceil$ and $\{\cdot\}$) denotes the floor function (resp. the ceiling function and the fractional part). For a finite set (resp. a finite word), we use $|\cdot|$ to denote its cardinality (resp. its length).    
\end{remark}

The idea behind the refined Diophantine exponent is to replace condition (ii) with a weaker condition that allows for an \emph{arbitrarily small} amount of mismatch between the two finite subwords. To state this explicitly, we need the following notion of \emph{$(\epsilon,\delta)$-closeness}.

\begin{definition} \label{Definition: epsilon delta close}
For two finite words $U=u_1\cdots u_L$ and $V=v_1\cdots v_L$ of the same length $L>0$, a real number $\epsilon>0$, and an integer $\delta\geq0$, we say that $U$ and $V$ are \emph{$(\epsilon,\delta)$-close} if there exist $\delta$ sub-intervals $I_1,\ldots,I_\delta$ of $\{1,\ldots,L\}$ such that:
\begin{itemize}\item[(i)]  we have $\left\{1\leq i\leq L:u_i\neq v_i\right\}\subseteq \bigcup_{j=1}^\delta I_{j};$
\item[(ii)] $\left|\bigcup_{j=1}^\delta I_{j}\right|\leq \epsilon L$.
\end{itemize}
 Note that the intervals $I_j$ can be empty.
\end{definition}

\begin{remark}
It follows that the \emph{Hamming distance} between $U$ and $V$, defined as $|\{1\leq i\leq L:u_i\neq v_i\}|$, is at most $\epsilon L$.
\end{remark}

\begin{definition}\label{Definition: refined Dio}
Let $\mathbf{a}=a_0a_1\cdots$ be an infinite word, and let $\rho\geq1$ be a real number. We say that $\mathbf{a}$ satisfies \emph{Condition $(*)_\rho$} if, for every $\epsilon>0$, there exist three sequences of integers $(r_n)_{n\geq 0}$, $(s_n)_{n\geq 0}$, $(t_n)_{n\geq 0}$ and  an integer $\delta\geq0$  satisfying conditions (i), (iii), and (iv) of Definition~\ref{Definition: Dio}, along with
\begin{itemize}
\item[(ii')]  such that for all $n$ large enough, $\mathbf{a}[r_n+1,r_n+t_n-s_n]$ is $(\epsilon,\delta)$-close to $\mathbf{a}[s_n+1, t_n]$.
\end{itemize}
The \emph{refined Diophantine exponent} $\mathbf{Rdio}(\mathbf{a})$ is defined to be the supremum of $\rho$ for which $\mathbf{a}$ satisfies Condition $(*)_\rho$.
\end{definition}

It is clear that $1\leq \mathbf{Dio}(\mathbf{a})\leq \mathbf{Rdio}(\mathbf{a})\leq\infty$. For further properties, we refer the reader to \cite{Khai-2026}. Notably, we have that $\mathbf{Rdio}(\mathbf{a})=1$ if and only if $\mathbf{Dio}(\mathbf{a})=1$. Furthermore, $\mathbf{Rdio}(\mathbf{a})=\infty$ when $\mathbf{a}\in$\{lacunary sequences, Sturmian words, $k$-bonacci word for $k\geq2$\}, whereas $\mathbf{Rdio}(\mathbf{a})=1$ when $\mathbf{a}$ arises from the $\beta$-expansion of an algebraic number in $[0,1)\setminus\mathbb{Q}(\beta)$ for any Pisot number $\beta$.

\subsection{Continued fractions}

The main reference for the theory is \cite[Chapter~1]{Bugeaud-ApproximationbyAlgebraicNumbers}. We collect here some necessary properties of continued fractions. 

Let $\theta$ be an irrational real number. We can express $\theta$ as an (infinite) continued fraction $[w_0; w_1, w_2, \ldots]$. Each $w_n$ is called a \emph{partial quotient} of $\theta$. The \emph{continued fraction convergents} of $\theta$ are the sequence \[\frac{p_n}{q_n} = [w_0; w_1, \ldots, w_n]\] where $q_n$ are strictly increasing for $n \geq 1$ and $p_n$ is coprime to $q_n$. We then have the fundamental inequality $|q_n\theta - p_n| < {1}/{q_{n+1}} < {1}/{q_n}$ for all $n \geq 0$. We write $\lVert\cdot\rVert$ for the distance of a real number to its nearest integer, so $\lVert q_n\theta\rVert=|q_n\theta-p_n|<{1}/{q_n}$.

The following is known as \emph{the law of best approximation}

\begin{proposition}\label{proposition: law of best approx}
 If $|q\theta - p| < |q_n\theta - p_n|$ for some $n \geq 1$ and some integers $p, q$ with $q \geq 1$, then $q \geq q_{n+1}$. In other words, amongst the fractions ${p}/{q}$ with $1 \leq q < q_{n+1}$, the one minimizing $|q\theta - p|$ is ${p_n}/{q_n}$.
\end{proposition}

\begin{definition}
 $\theta$ is said to be \emph{badly approximable} if $\theta$ has bounded partial quotients; otherwise, $\theta$ is called \emph{well approximable}. 
\end{definition}

\begin{proposition}
    Let $\theta$ be an irrational number with continued fraction convergents ${p_n}/{q_n}$. The following assertions are equivalent:
    \begin{itemize}
        \item[(i)] There exists $\kappa > 0$ such that $|q_n\theta - p_n| \geq {\kappa}/{q_n}$ for all $n$.
         \item[(ii)] There exists $\kappa > 0$ such that $|q\theta - p| \geq {\kappa}/{q}$ for all integers $p, q$ with $q > 0$.
         \item[(iii)] There exists $A > 0$ such that $q_{n+1} \leq A q_n$ for all $n$.
         \item[(iv)] $\theta$ is badly approximable.
    \end{itemize}
\end{proposition}

\section{Spectrum of the refined Diophantine exponent}\label{section: spectrum of Rdio}

Our approach is motivated by seminal works on the pseudorandomness of finite words \cite{Mauduit-Sarkozy-1997-I,Mauduit-Sarkozy-1998-II,Mauduit-1999-III,Mauduit-1999-IV,Cassaigne-Mauduit-Sarkozy-2002-VII}. To measure the pseudorandomness, they use \emph{correlations} and study them from both probabilistic and arithmetic viewpoints (e.g. Champernowne, Rudin--Shapiro, Thue--Morse, and words arising from the Legendre symbol, etc). 

\begin{definition}\label{definition: Cbmkn Dbmkn}
For an infinite word $\mathbf{b}=b_0b_1\cdots$ over the alphabet $\{-1,+1\}$ and integers $m\geq0,k\geq1,N\geq1$, we set
\[C(\mathbf{b};m,k,N)=\left|\sum_{i=m}^{m+N-1}b_ib_{i+k}\right|\in[0,N].\]
We also set \[D(\mathbf{b};m,k,N) = \sum_{i=m}^{m+N-1} \frac{1 - b_i b_{i+k}}{2},\]which is  the exact number of mismatches between $\mathbf{b}[m,m+N-1]$ and $\mathbf{b}[m+k,m+N-1+k]$ (also known as the Hamming distance).
\end{definition}

\begin{remark}
    The $k$-correlation of $\mathbf{b}$ is known as the limit \[\displaystyle\lim_{N\to\infty}\frac{C(\mathbf{b};0,k,N)}{N}\in[0,1].\] Following \emph{loc. cit.}, \emph{pseudorandomness} means that the $k$-correlation is small relative to $N$, that is, the exact number of mismatches is large relative to $N$.
\end{remark}

To prove Theorem~\ref{Theorem: spectrum of Rdio}, we first find a suitable \emph{pseudorandom} word $\mathbf{b}$, then modify this word by using a lacunary sequence with a specified Diophantine exponent. We then show that the modified word inherits the desired refined Diophantine exponent. However, the results in \emph{loc. cit.} are not good enough for our purpose. For our construction in Theorem~\ref{Theorem: spectrum of Rdio}, we require the number of mismatches to be linearly large in $N$. More precisely, for $m, k = O(N)$, we need $D(\mathbf{b}; m, k, N) \gg N$. Translating this to $C(\mathbf{b}; m, k, N)$, we seek an infinite binary word subject to the following slightly stronger condition.

\begin{theorem}\label{theorem: existence of word}
There exists an infinite word $\mathbf{b}$ over $\{-1,+1\}$ such that 
\[C(\mathbf{b};m,k,N)\leq 4\sqrt{(m+k+N)\log(m+k+N)}\]
for all  $m\geq0,k\geq1,N\geq1$.
\end{theorem}
\begin{proof}
We use the probabilistic method. We equip the set $\Omega=\{-1,+1\}^{\mathbb N}$ with the uniform Bernoulli product measure, meaning each letter $b_i$ is an independent random variable taking values $-1$ or $+1$ with probability ${1}/{2}$. We would like to apply the Azuma--Hoeffding inequality (see e.g. \cite[ Section~3.2]{Roch-2024}) to the stochastic process $\{X_0,X_1,\ldots\}$ defined by $X_0 = 0$ and \[ X_i = X_{i-1} + b_{m+i-1} b_{m+i-1+k}, \] with respect to the filtration \[\{\emptyset, \Omega\}=\mathcal{F}_{-1}\subset  \mathcal{F}_0\subset \mathcal{F}_1\subset\ldots\] where each $\mathcal{F}_i,i\geq0$, is defined as the smallest $\sigma$-algebra generated by $b_0, b_1,\ldots,b_{m+i-1+k}$. 

First, we claim that $X_1,X_2,\ldots$ is a martingale with respect to $\mathcal{F}_i$. Indeed, since $|X_i| \leq i$, we have \[ \mathbb{E}[|X_i|] \leq i < \infty. \]
In addition, we have
\begin{align*}
    \mathbb{E}[X_{i+1} \mid \mathcal{F}_i] &= \mathbb{E}[X_i + b_{m+i} b_{m+i+k} \mid \mathcal{F}_i]\\&= \mathbb{E}[X_i\mid \mathcal{F}_i]+ \mathbb{E}[b_{m+i} b_{m+i+k}\mid \mathcal{F}_i]\\&=\mathbb{E}[X_i\mid \mathcal{F}_i]+ b_{m+i}\mathbb{E}[ b_{m+i+k}\mid \mathcal{F}_i]\\&=X_i 
\end{align*}
since $\mathbb{E}[ b_{m+i+k}\mid \mathcal{F}_i]=0$. Thus $X_1,X_2,\ldots$ is a martingale.

Now we apply the Azuma--Hoeffding inequality to $X_1,\ldots,X_{N}$ to deduce that
\begin{align*}
\mathbb P(C(\mathbf{b};m,k,N)\geq 4\sqrt{(m+k+N)\log(m+k+N)})&\leq 2e^{\frac{-16(m+k+N)\log(m+k+N)}{2N}}\\&\leq 2(m+k+N)^{-8}.    
\end{align*}
 Here, we have used that $|X_{i+1}-X_{i}|\leq 1$ for all $i$. In addition, we note that for a given $L\geq2$, the number of triples $(m,k,N)$ such that $m+k+N=L$ is  ${L(L-1)}/{2}<L^2$, therefore 
\[\sum_{m,k,N}\mathbb P(C(\mathbf{b};m,k,N)\geq 4\sqrt{(m+k+N)\log(m+k+N)})\leq2\sum_{L\geq2}L^{-6}<1.\]
Using the union bound, it follows that there is a word $\mathbf{b}$ over $\{-1,+1\}$ such that \[C(\mathbf{b};m,k,N)\leq 4\sqrt{(m+k+N)\log(m+k+N)}\] for all $k,m,N$ as desired.
\end{proof}

\begin{remark}\label{remark: lower bound for Dbmkn}
    If we set $E(\mathbf{b};m,k,N)=|\{i\in[m,m+N-1]:b_i=b_{i+k}\}|$, then \[E(\mathbf{b};m,k,N)+D(\mathbf{b};m,k,N)=N\] and\begin{align*}
        |E(\mathbf{b};m,k,N)-D(\mathbf{b};m,k,N)|&=C(\mathbf{b};m,k,N)\\&\leq 4\sqrt{(m+k+N)\log(m+k+N)}.
    \end{align*}  It follows that \[D(\mathbf{b};m,k,N)\geq \frac{N}{2}-2\sqrt{(m+k+N)\log(m+k+N)}.\] In the next proof, we will apply this inequality to the following situation: $N$ runs over an infinite set of positive integers, the growth of $m$ and $k$ are $O(N)$, so $D(\mathbf{b};m,k,N)\gg N$ where the implied constant is independent of $N$. 
\end{remark}

Now we show that the spectrum of $\mathbf{Rdio}$ is $[1,\infty]$.
\begin{proof}[Proof of Theorem~\ref{Theorem: spectrum of Rdio}]
 It suffices to show that for every $C>1$, there exists an infinite word $\mathbf{a}$ over $\{-1,0,1\}$ such that $\mathbf{Rdio}(\mathbf{a})={(C+1)}/{2}$.

Let $\mathbf{b}$ be the word in Theorem \ref{theorem: existence of word}.  We consider the sequence $u_i=\lfloor C^i\rfloor$, and the word \[\mathbf{a}=0^{u_1}b_0\cdots b_{u_{1}-1}0^{u_2}b_{u_1}\cdots b_{u_1+u_2-1}0^{u_3}b_{u_1+u_2}\cdots b_{u_1+u_2+u_3-1}\cdots\] built from a block of $0$ of length $u_1$, then finite subword formed by the first $u_1$ letters of $\mathbf{b}$, then a block of $0$ of length $u_2$, and so on. Then it is clear that \[\mathbf{Rdio}(\mathbf{a})\geq\mathbf{Dio}(\mathbf{a}) \geq \displaystyle\lim_{l \to \infty} \left( 1 + \frac{u_l}{2 \sum_{i=1}^{l-1} u_i} \right)=\frac{C+1}{2}>1.\]

    Next, we show that $\mathbf{Rdio}(\mathbf{a})\leq {(C+1)}/{2}$.  Let $\rho>1$ be any real number such that $\rho<\mathbf{Rdio}(\mathbf{a})$. Then for every $\epsilon>0$, there exist sequences $(r_n),(s_n),(t_n)$ and $\delta\geq0$ with respect to $\rho$ as in the definition of $\mathbf{Rdio}(\mathbf{a})$. Remark that for the $s_n-r_n$-shift in $\mathbf{a}$, the corresponding shift in $\mathbf{b}$ is given by shifting $v\leq s_n-r_n$ letters, depending on the position of the letter in $\mathbf{b}$.
  
    We observe that for infinitely many $n$, if $\mathbf{a}[r_n+1,r_n+t_n-s_n]$ contains a whole block of finite words of $\mathbf{b}$, then there is some maximal $l$ such that $b_{u_1+u_2+\cdots+ u_{l-1}}\cdots b_{u_1+u_2+\cdots+u_l-1}$ belongs to $\mathbf{a}[r_n+1,r_n+t_n-s_n]$. In this case, we would like to give a lower bound for the number of mismatches. Considering its shift by $s_n-r_n$ in $\mathbf{a}$, there are four cases:
    \begin{itemize}
        \item Its shift lies entirely inside a 0-block;
        \item Its shift lies entirely inside a $\mathbf b$-block;
        \item Its shift lies inside a 0-block followed by a $\mathbf b$-block;
        \item Its shift lies entirely inside a $\mathbf b$-block followed by a 0-block.
    \end{itemize}
    Therefore, to give a lower bound, it suffices to consider the worst case, which is that its shift lies entirely inside a $\mathbf b$-block. Under this assumption, the corresponding shift in $\mathbf{b}$ is some $v\leq s_n-r_n$. It follows that
    \[(\rho-1)s_n\leq t_n-s_n\leq 2\left(\sum_{i=1}^{l+1}u_i\right)\ll u_l,\]
    while by Remark~\ref{remark: lower bound for Dbmkn}, the number of mismatches is at least
    \begin{align*}
        \frac{u_l}{2}-&2\sqrt{(u_1+\cdots+ u_{l-1}+v+u_{l+1})\log(u_1+\cdots+ u_{l-1}+v+u_{l+1})}\\& \gg u_{l},
    \end{align*}
    where the implied constants are independent of $\epsilon$. This contradicts the $(\epsilon,\delta)$-closeness, see  Definition  \ref{Definition: epsilon delta close}. Therefore, when $n$ is large enough, $\mathbf{a}[r_n+1,r_n+t_n-s_n]$ belongs to one of the following forms (where a partial block is allowed to be empty):
    \begin{itemize}
     \item It contains one (partial) block of $0$ and followed by one partial block of finite subwords of $\mathbf{b}$;
     \item It contains one partial block of finite subwords of $\mathbf{b}$ and followed by one (partial) block of $0$;
        \item It contains two partial blocks of finite subwords of $\mathbf{b}$ and a whole block of $0$s separating them.
    \end{itemize}
    In addition, if for infinitely many $n$, one of the blocks of finite words of $\mathbf{b}$ occurring in $\mathbf{a}[r_n+1,r_n+t_n-s_n]$, say $b_{u_1+u_2+\cdots+u_{l-1}+p}\cdots b_{u_1+u_2+\cdots+u_{l-1}+q}$, does not have length $o(u_l)$, then there exists some $c>0$ such that $q-p\geq c u_l$ for infinitely many $n$ and $l$. Then, for some $v\leq s_n-r_n$, the number of mismatches is at least the number of $i,u_1+u_2+\cdots+p\leq i\leq u_1+u_2+\cdots+q$, such that $b_i\neq b_{i+v}$. As observed above, there are at most two partial blocks of finite words of $\mathbf{b}$. Again, we have \[(\rho-1)s_n\leq t_n-s_n\leq 2\left(\sum_{i=1}^{l+1}u_i\right)\ll u_l,\] while by Remark~\ref{remark: lower bound for Dbmkn}, the number of mismatches is at least
    \begin{align*}
        \frac{q-p}{2}-&2\sqrt{(u_1+\cdots+ u_{l-1}+v+2u_{l+1})\log(u_1+\cdots+ u_{l-1}+v+2u_{l+1})}\\& \gg u_l,
    \end{align*}
 which is absurd. Furthermore, if $\mathbf a[r_n+1,r_n+t_n-s_n]$ is contained entirely in a
$0^{u_\ell}$-block, then
\[
\frac{t_n}{s_n}
\le
1+\frac{u_\ell+O(1)}
{2\sum_{i=1}^{\ell-1}u_i+O(1)}
\to \frac{C+1}{2}.
\]
If it is contained entirely in a $b$-block, then the preceding argument
shows that $t_n-s_n=o(u_\ell)$, while $s_n\gg u_\ell$, and hence
$t_n/s_n\to1$, which is absurd. Thus it remains to consider the following three cases.
    \begin{itemize}
        \item If for infinitely many $n$, we are in the following situation: $a_{r_n+1}=0$ belongs to $0^{u_l}$ and $a_{r_n+t_n-s_n}$ belongs to $b_{u_1+\cdots+u_{l-1}}\cdots b_{u_1+\cdots+u_{l}-1}$, then \[2\sum_{i=1}^{l-1}u_i\leq r_n\leq2\sum_{i=1}^{l-1}u_i+u_l-1\] and\begin{align*}
            \sum_{i=1}^{l-1}u_i+u_l&\leq r_n+t_n-s_n\\&\leq 2\sum_{i=1}^{l-1}u_i+u_l+o(u_l)\\&\leq 2\sum_{i=1}^{l}u_i-1.
        \end{align*}
        It follows that $s_n> r_n\geq 2\sum_{i=1}^{l-1}u_i-1$ and \begin{align*}
        \frac{t_n}{s_n}&\leq \frac{s_n+2\sum_{i=1}^{l-1} u_i+u_l+o(u_l)-r_n}{s_n}\\&\leq 1+\frac{u_l+o(u_l)}{2\sum_{i=1}^{l-1}u_i-1}.    
        \end{align*}
       We note that \[\lim_{l\to\infty }\left(1+\frac{u_l+o(u_l)}{2\sum_{i=1}^{l-1}u_i-1}\right)=\frac{C+1}{2}.\]
        \item If for infinitely many $n$, we are in the following situation:  $a_{r_n+1}$ belongs to $b_{u_1+\cdots+u_{l-1}}\cdots b_{u_1+\cdots+u_{l}-1}$ and $a_{r_n+t_n-s_n}$ belongs to $0^{u_{l+1}}$, then\begin{align*}
            2\sum_{i=1}^{l-1}u_i+u_l\leq r_n+1&=2\sum_{i=1}^{l-1}u_i+2u_l-o(u_l)\\&\leq2\sum_{i=1}^{l-1}u_i+2u_l-1,
        \end{align*}  and  \[2\sum_{i=1}^{l}u_i\leq r_n+t_n-s_n\leq 2\sum_{i=1}^{l}u_i+u_{l+1}-1.\]
         It follows that  $s_n>r_n\geq 2\sum_{i=1}^{l-1}u_i+u_l-1$ and 
        \[\frac{t_n}{s_n}\leq \frac{s_n+2\sum_{i=1}^l u_i+u_{l+1}-r_n}{s_n}\leq 1+\frac{u_{l+1}+o(u_l)}{2\sum_{i=1}^{l}u_i+o(u_l)}.\]
      We also have that \[\lim_{l\to\infty }\left(1+\frac{u_{l+1}+o(u_l)}{2\sum_{i=1}^{l}u_i+o(u_l)}\right)=\frac{C+1}{2}.\]
        \item If for infinitely many $n$, we are in the following situation:  $a_{r_n+1}$ belongs to $b_{u_1+\cdots+u_{l-1}}\cdots b_{u_1+\cdots+u_{l}-1}$ and $a_{r_n+t_n-s_n}$ belongs to \[b_{u_1+\cdots+u_{l}}\cdots b_{u_1+\cdots+u_{l}+u_{l+1}-1}\] then
        \begin{align*}
        2\sum_{i=1}^{l-1}u_i+u_l\leq r_n+1&=2\sum_{i=1}^{l}u_i-o(u_l)\\&\leq  2\sum_{i=1}^{l}u_i-1 
        \end{align*}
        and 
        \begin{align*}
            2\sum_{i=1}^{l}u_i+u_{l+1}&\leq r_n+t_n-s_n\\&\leq 2\sum_{i=1}^{l}u_i+u_{l+1}+o(u_{l+1})\\&\leq 2\sum_{i=1}^{l+1}u_i-1.
        \end{align*}
      It follows that 
      \begin{align*}
          \frac{t_n}{s_n}&\leq \frac{s_n+2\sum_{i=1}^{l}u_i+u_{l+1}+o(u_{l+1})-r_n}{s_n}\\&\leq 1+\frac{u_{l+1}+o(u_{l+1})+o(u_l)}{2\sum_{i=1}^l u_i-o(u_l)}.
      \end{align*}  
         Again, we have  \[\lim_{l\to\infty }\left(1+\frac{u_{l+1}+o(u_{l+1})+o(u_l)}{2\sum_{i=1}^l u_i-o(u_l)}\right)=\frac{C+1}{2}.\]
    \end{itemize}
  
  Combining these three cases, we deduce that $\rho\leq {(C+1)}/{2}$ for all $\rho<\mathbf{Rdio}(\mathbf{a})$. Therefore, $\mathbf{Rdio}(\mathbf{a})\leq{(C+1)}/{2}$. We conclude that $\mathbf{Rdio}(\mathbf{a})={(C+1)}/{2}$ as desired. 
\end{proof}

In view of \cite[Proposition~2.24]{Khai-2026}, we can modify the word $\mathbf{a}$ constructed in the previous proof to obtain the following result. The proof is a combination of the argument in \emph{ibid.} and the previous proof, so we omit it.
 
\begin{corollary}
    Over a quaternary base, for a given $C>1$, there exists an infinite word $\mathbf{a}$ such that $\mathbf{Dio}(a)<\mathbf{Rdio}(a)=C$.
\end{corollary}

This corollary partly explains why \cite[Theorem~A]{Khai-2026} provides a new and stronger combinatorial transcendence criterion than those in \cite{Adamczewski-Bugeaud-2007-dynamics,Luca-Ouaknine-Worrell-2025,Luca-Ouaknine-Worrell-2023,Kebis-Luca-Ouaknine-Scoones-Worrell-2024,Kebis-Luca-Ouaknine-Scoones-Worrell-2025}, since it applies to a broader family of words.

\section{Topological properties of the refined Diophantine exponent}\label{section: topo properties}
We start this section with a proof of the first assertion of Theorem~\ref{Theorem: probability of Rdio=1}, which says that $\mathbb{P}(\mathbf{Rdio}(\mathbf{a}) = 1) = 1$.
\begin{proof}[Proof of the first assertion of Theorem~\ref{Theorem: probability of Rdio=1}]
  Since $\mathbf{Rdio}(\mathbf{a})=1$ if and only if $\mathbf{Dio}(\mathbf{a})=1$ by \cite[Proposition~2.1]{Khai-2026}, it suffices to show that $\mathbb{P}(\mathbf{Dio}(\mathbf{a}) > 1) = 0$. We will use the Borel--Cantelli lemma. 

    Let $q = |\Sigma| \geq 2$. Let $\eta>0$ be a rational number.  For any $-1\leq r<s<t$, we have
\[\mathbb{P}(\mathbf{a}[r+1, r+t-s] = \mathbf{a}[s+1, t]) \leq q^{s-t}.\]

For each $s\geq0$, let $A_s$ be the event that for the index $s$, there exists some $r \in [-1, s-1]$ and some $t \geq (1+\eta) s$ such that $\mathbf{a}[r+1, r+t-s] = \mathbf{a}[s+1, t]$. For each $s\geq0$, we apply the union bound over all $s+1$ choices for $r$ and all possible choices of $t$ to deduce that \[\mathbb{P}(A_s) \leq \sum_{r=-1}^{s-1} \sum_{t = \lceil \eta s \rceil+s}^{\infty} q^{s-t} \leq (s+1) \left( \frac{q^{-\eta s}}{1 - q^{-1}} \right).\]
    It follows that
    \[\sum_{s=0}^{\infty} \mathbb{P}(A_s) \leq \frac{1}{1 - q^{-1}} \sum_{s=0}^{\infty} (s+1) (q^{-\eta})^s.\]
   Note that $q^{-\eta} < 1$. Therefore, \[\sum_{s=1}^{\infty} \mathbb{P}(A_s) < \infty.\]
   By the Borel--Cantelli Lemma, the probability that the event $A_s$ happens for infinitely many $s$ is exactly $0$.

Now, taking the countable union over all rational numbers $\eta$, we obtain $\mathbb{P}(\mathbf{Dio}(\mathbf{a}) > 1) = 0$ as wanted. Here, recall that for an infinite word $\mathbf{a}$, if $\mathbf{Dio}(\mathbf{a}) > 1$, then there exist a rational number $\eta > 0$ and sequences of integers $(r_n)$, $(s_n)$, $(t_n)$ satisfying conditions in Definition~\ref{Definition: Dio} with $\rho=1+\eta$. In particular, $\mathbf{a}$ satisfies the event $A_{s_n}$ for all $n$ sufficiently large as desired.
\end{proof}

For the second assertion, we need the following \emph{shift-invariant} property.

\begin{proposition}\label{proposition: Rdio is shift invariant}
Let $\mathbf{a}=a_0a_1\cdots$ be an infinite word and $m$ be a positive integer. We set $\mathbf{a}'=a_ma_{m+1}\cdots$. Then $\mathbf{Rdio}(\mathbf{a})=\mathbf{Rdio}(\mathbf{a}')$. 
\end{proposition}

\begin{proof}
    We first show that $\mathbf{Rdio}(\mathbf a')\ge \mathbf{Rdio}(\mathbf a).$ There is nothing to prove if $\mathbf{Rdio}(\mathbf a)=1$. Thus, let
\[
1<\rho<\rho_1<\mathbf{Rdio}(\mathbf a).
\]
Fix $\epsilon>0$. With $\epsilon/2$,
by definition there are sequences $(r_n),(s_n),(t_n)$ and an integer $\delta\ge 0$
such that $t_n\ge \rho_1s_n, s_n-r_n\to\infty,$
and, for all sufficiently large $n$, $\mathbf a[r_n+1,r_n+t_n-s_n]$ and $\mathbf a[s_n+1,t_n]$ are $(\epsilon/2,\delta)$-close.

Passing to a subsequence if necessary, we may assume that either
$r_n\ge m$ for every $n$, or $r_n<m$ for every $n$.
\begin{itemize}
    \item The first case: $r_n\ge m$. We then set $r_n'=r_n-m$, $s_n'=s_n-m$ and $t_n'=t_n-m.$ Then $\mathbf a'[r_n'+1,r_n'+t_n'-s_n']
=\mathbf a[r_n+1,r_n+t_n-s_n]$
and $\mathbf a'[s_n'+1,t_n']
=\mathbf a[s_n+1,t_n].$ Moreover, we have $\frac{t_n'}{s_n'}=\frac{t_n-m}{s_n-m}\ge \frac{t_n}{s_n}
\ge \rho_1>\rho,$
and $s_n'-r_n'=s_n-r_n\to\infty$. Hence these data satisfy $(*)_\rho$ for $\mathbf a'$.
\item The second case: $r_n<m$. Since $-1\le r_n\le m-1$, after passing to a further subsequence we may assume that $r_n=r$ is constant. We set $h:=m-r-1\ge 0$ and $r_n'=-1$, $s_n'=s_n-r-1$, $t_n'=t_n-m.$ Then $\mathbf a'[r_n'+1,r_n'+t_n'-s_n']
=\mathbf a[m,r+t_n-s_n]$
and $\mathbf a'[s_n'+1,t_n']
=\mathbf a[s_n+h+1,t_n].$
Since $\rho_1>1$ and $s_n\to\infty$, we have $t_n-s_n\to\infty$. Hence, for all sufficiently large $n$, $t_n-s_n\ge 2h$. It follows that $\frac{\epsilon}{2}(t_n-s_n)
\le \epsilon(t_n-s_n-h).$
Therefore, the resulting two words are $(\epsilon,\delta)$-close. Furthermore, we have $
\frac{t_n'}{s_n'}=\frac{t_n-m}{s_n-r-1}\ge\frac{\rho_1s_n-m}{s_n-r-1}
\to \rho_1.$ Hence, for all sufficiently large $n$, we have $t_n'\ge \rho s_n'$ and
$s_n'-r_n'=s_n-r\to\infty.$
Thus $\mathbf a'$ satisfies $(*)_\rho$.
\end{itemize}
In any case, $\mathbf a'$ always satisfies $(*)_\rho$. Therefore $\mathbf{Rdio}(\mathbf a')
\ge \mathbf{Rdio}(\mathbf a).$ A symmetric argument yields that  $\mathbf{Rdio}(\mathbf{a}')\leq\mathbf{Rdio}(\mathbf{a})$, hence finishing the proof.
\end{proof}

\begin{remark}
      The Diophantine exponent is also shift-invariant, see \cite[Proposition~6.4]{Peltomaki-2024}.
\end{remark}

\begin{proof}[Proof of the second assertion of Theorem~\ref{Theorem: probability of Rdio=1}]
 Let $c\in[1,\infty]$. To prove density in the Cantor topology, we must show that for an arbitrary finite word $w$, there exists at least one infinite word $\mathbf{a} \in [w]$ such that $\mathbf{Rdio}(\mathbf{a}) =c $.

    By Theorem~\ref{Theorem: spectrum of Rdio}, there is $\mathbf{a}'$ such that $\mathbf{Rdio}(\mathbf{a}') = c$. We then set $\mathbf{a}=w\mathbf{a}'$, by Proposition~\ref{proposition: Rdio is shift invariant} we have  \[\mathbf{Rdio}(\mathbf{a})=\mathbf{Rdio}(\mathbf{a}') = c.\]
    The theorem follows.
\end{proof}

\begin{remark}
   The proof also shows that for any alphabet $\Sigma$ with $|\Sigma| \geq 2$, the set of infinite words over $\Sigma$ whose refined Diophantine exponent equals $1$ (resp., $\infty$) is dense in the Cantor topology.
\end{remark}

\section{The Champernowne word}\label{section: Champernowne}
In this section, we show that the $\mathbf{Rdio}$ (and $\mathbf{Dio}$) of the {Champernowne word} is equal to $1$. Recall that the Champernowne word $\mathbf{c}$ is an infinite word defined by sequentially concatenating the representations of consecutive integers $1, 2, \ldots$ in base $10$. Here are its first letters
\[\mathbf{c}=1234567891011121314151617181920\cdots\]

It is the most famous example of a \emph{normal} word. This means that every possible finite block of digits of length $k$ appears in the sequence with a uniform frequency of $10^{-k}$. Because it is normal, it contains every possible finite sequence of digits. Consequently, its subword complexity is maximal, which intuitively suggests that its $\mathbf{Rdio}$ should be small, even equal to $1$. However, one can manually construct a normal word whose $\mathbf{Rdio}$ is infinite. In fact, the proof of the following theorem does not use the normality of the Champernowne 
word.

\begin{theorem}We have $\mathbf{Rdio}(\mathbf{c})=\mathbf{Dio}(\mathbf{c})=1$.
\end{theorem}
\begin{proof}
By \cite[Proposition 2.1]{Khai-2026}, it suffices to show that $\mathbf{Dio}(\mathbf{c})=1$. Suppose, for contradiction, that
$\mathbf{Dio}(\mathbf c)>1.$
We choose $1<\rho<\mathbf{Dio}(\mathbf c).$ By definition, there exist integers $-1\le r_n<s_n<t_n$
with $s_n-r_n\to\infty$ such that
$\mathbf c[r_n+1,r_n+t_n-s_n]=\mathbf c[s_n+1,t_n]
$ and $t_n\ge \rho s_n.$ Hence $t_n-s_n\ge(\rho-1)s_n.$ 

For an integer $N\ge1$, let $[N]$ denote its decimal representation.

For $D\ge1$, let $W_D$ denote the finite subword of $\mathbf c$ obtained by concatenating all $D$-digit integers
$$W_D=[10^{D-1}][10^{D-1}+1]\cdots[10^D-1].$$
Its length is $$|W_D|=9D10^{D-1}.$$
We let $$A_D=\sum_{j=1}^{D-1}9j10^{j-1},$$
so that $A_D$ is the position where the $D$-digit zone starts. Note that $A_D\ge 9(D-1)10^{D-2}.$ 

Now, we choose $0<\eta<\min\{\rho-1,1/10\}$. For sufficiently large $s_n$,
 we have $\lfloor\eta s_n\rfloor\leq t_n-s_n.$ Therefore $\mathbf c[s_n+1,s_n+\lfloor\eta s_n\rfloor]=\mathbf c[r_n+1,r_n+\lfloor\eta s_n\rfloor].$ 

 Let $K_n$ be such that $c_{s_n}$ belongs to the $K_n$-digit zone. Then $A_{K_n}\le s_n<A_{K_n+1}.$ Note that when $s_n\to\infty$, we have $K_n\to\infty$. Moreover, we have $s_n<A_{K_n+1}<K_n10^{K_n}
$ and $\lfloor\eta s_n\rfloor\le s_n/10$, so $$s_n+\lfloor\eta s_n\rfloor<1.1K_n10^{K_n}.$$
On the other hand, we have 
$$A_{K_n+2}\ge9(K_n+1)10^{K_n}.$$
Thus, for large $K_n$, the interval
$\mathbf{c}[s_n+1,s_n+\lfloor\eta s_n\rfloor]$ meets only $W_{K_n}$ and possibly $W_{K_n+1}$. Hence one of these two intersections contains a finite subword $\mathbf{c}_n'\subset W_{D_n}$, where
$D_n\in\{K_n,K_n+1\},$
such that $|\mathbf{c}_n'|\ge\frac {\lfloor\eta s_n\rfloor}{2}.$ For large $s_n$ we also have $\lfloor\eta s_n\rfloor \ge\eta s_n/2$, so
$$|\mathbf{c}_n'|\ge\frac{\eta s_n}{4}\ge\frac{9\eta}{4}(K_n-1)10^{K_n-2}.
$$
It follows that there exists some $\gamma>0$, depending only on
$\rho$, such that
$$|\mathbf{c}'_n|\ge\gamma D_n10^{D_n-1}$$
for all sufficiently large $K_n$. 

We shall prove that for $D$ sufficiently large, every finite subword $\mathbf{c}_0$ satisfying $|\mathbf{c}_0|\ge\gamma D10^{D-1}$ contains a word whose natural occurrence is its first occurrence in $\mathbf c$. Here and from now on in this proof, by \emph{natural occurrence} we mean the appearance as its decimal representation in $\mathbf{c}$. We fix an integer $\ell\ge1$, to be chosen later, and assume $D>2\ell.$ For every integer $a$ satisfying $10^{\ell-1}<a<10^\ell,$ we define
$$M_D(a)=[a-1]9^{D-\ell}[a]0^{D-\ell}=[a10^{D-\ell}-1][a10^{D-\ell}].$$
Thus $M_D(a)$ has a natural occurrence in $W_D$. 

\begin{claim}
    For all but at most $4\ell(\ell+1)$ of these $M_D(a)$ has no occurrence before its natural occurrence. 
\end{claim}
\begin{proof}[Proof of the claim]
    First suppose that an occurrence of $M_D(a)$ is contained entirely in a
zone $W_d$ with $d<D$. We can not have $d\le D-\ell,$ because $M_D(a)$ contains $0^{D-\ell}$, whereas a block of
$D-\ell$ consecutive zeros cannot occur in $W_d$. Now we consider the case $D-\ell<d<D.$ We write $d=D-j$ where $1\le j<\ell.$ Let $z$ be the $d$-digit integer containing the block $0^{D-\ell}$. Write $[z]=P0^{D-\ell}S,$
where $|P|=h.$ Since $d-(D-\ell)=\ell-j,$
we have $|S|=\ell-j-h.$ Since the first digit of $[z]$ is nonzero, we have $1\le h\le\ell-j.$

We claim that $S=0^{|S|}.$ Suppose otherwise, we then have $[z-1]=P0^{D-\ell}[S-1]$. Since
$\ell-h=j+|S|$ and $D-\ell>\ell-h$, the last $\ell-h$ digits of $[z-1]$ begin with $j$ zeros. Hence the $\ell$ digits preceding $0^{D-\ell}$, which form
$[a]$, begin with $0$, a contradiction.

Therefore $[z]=P0^{d-h}$ and $[z-1]=(P-1)9^{d-h}.$ Moreover, the $\ell$ digits preceding the block $0^{D-\ell}$
are $[a]$, so $[a]=9^{\ell-h}P.$

Now, recall that $M_D(a)$ contains $9^{D-\ell}$ before $[a]$. Since $[a]$ begins with $9^{\ell-h}$, there are ${D-\ell}+\ell-h=D-h$
consecutive $9$'s before the block $P$. If $h\le j$, then $D-h\ge D-j=d.$
Hence the $d$-digit integer $z-1$ would consist of only $9$'s, so $z-1=10^d-1
$ and therefore $z=10^d,$ contradicting the fact that $z$ is a $d$-digit integer. Thus $h>j.$

Since $[z-1]$  ends in $d-h$ nines, the last $j$ digits of $P-1$ must also
be $9$'s. Hence $P=Q0^j$ for some integer $Q$. We deduce that $[a]=9^{\ell-h}Q0^j.$

Next, we determine $[a-1]$. The block $9^{D-\ell}$ starts inside $[z-1]$ after
the initial block $Q-1$, whose length is $h-j<\ell$. Since $[a-1]$
has length $\ell$ and precedes $9^m$, the beginning of $[a-1]$ must lie in $[z-2]$. As the whole occurrence lies inside $W_d$ by assumption, the integer $z-2$
is also a $d$-digit integer. From
$[z]=Q0^{D-h}$ we obtain $[z-1]=(Q-1)9^{D-h}$ and $[z-2]=(Q-1)9^{D-h-1}8.$ Therefore $[a-1]=9^{\ell-h+j-1}8(Q-1).$ It follows that $a=10^\ell-10^h+Q10^j$ and $a-1=10^\ell-2\cdot10^{h-j}+Q-1$. Therefore, $Q(10^j-1)
=10^{h-j}(10^j-2).$ Since $\gcd\left(10^j-1,10^{h-j}(10^j-2)
\right)=1,$ we deduce that $10^j-1=1,$
which is absurd for $j\ge1$.

Thus $M_D(a)$ cannot occur entirely inside any $W_d$ with $d<D$.

Next suppose that an occurrence of $M_D(a)$ is contained entirely in
$W_D$. Recall that $M_D(a)=[a10^{D-\ell}-1][a10^{D-\ell}].$ We claim that such an occurrence must be aligned with the appearance of $D$-digit integers.

Indeed, suppose that it starts $q$ digits after the beginning of some
$D$-digit integer $y$, where
$1\le q<D.$ Note that $y\bmod10^{D-q}$ is the last $D-q$ digits of $y$, while $\left\lfloor\frac{y+1}{10^{D-q}}\right\rfloor $ is the first $q$ digits of $y+1$. Thus, the first two $D$-digit blocks from this position are
$$B=(y\bmod 10^{D-q})10^q+\left\lfloor\frac{y+1}{10^{D-q}}\right\rfloor
$$
and
$$B'=((y+1)\bmod 10^{D-q})10^q+
\left\lfloor\frac{y+2}{10^{D-q}}\right\rfloor.$$
It follows that
$$B'-B=
\begin{cases}
10^q,&y\bmod 10^{D-q}\le 10^{D-q}-3,\\
10^q+1,&y\bmod 10^{D-q}=10^{D-q}-2,\\
-(10^{D-q}-1)10^q,&y\bmod 10^{D-q}=10^{D-q}-1.
\end{cases}
$$
Since $q\ge1$, none of these numbers is equal to $1$. This is
absurd because the two $D$-digit blocks of $M_D(a)$ are $a10^{D-\ell}-1$ and
$a10^{D-\ell}$.

We have therefore proved that any occurrence of $M_D(a)$ before its
natural occurrence must cross a transition between two digit-length zones.

Suppose such an occurrence crosses the transition from the
$(e-1)$-digit zone to the $e$-digit zone. Since $M_D(a)$ ends with
$0^{D-\ell}$, this block $0^{D-\ell}$ must lie after the transition.
Moreover, it cannot cross a boundary between two $e$-digit integers,
since every $e$-digit integer begins with a nonzero digit. Hence $0^m$
is contained in a single $e$-digit integer. It follows that $e\ge D-\ell+1.$ Since this occurrence is before the natural occurrence in $W_D$, we also have $e\le D.$
Thus there are at most $\ell$ possible values of $e$.

Furthermore, we have $e\ge D-\ell+1>\frac D2.$ Since $|M_D(a)|=2D$, the $e$-digit integer containing the terminal $0^{D-\ell}$ must be among the first four $e$-digit integers since $\frac{2D}{e}<4$. For each such integer, the block $0^{D-\ell}$ can start in at most $e-(D-\ell)+1\le\ell+1$ positions. Once the position of this block $0^{D-\ell}$ is determined, the starting position of the $M_D(a)$ is determined, and therefore $a$ is determined. It follows that the number of $a$ for which an earlier occurrence can cross a transition between two zones is at most $4\ell(\ell+1)$ as desired. In other words, all but at most $4\ell(\ell+1)$ of $M_D(a)$ have no occurrence before their natural occurrence as desired.
\end{proof}

For convenience, we call $M_D(a)$ \emph{good} if its natural occurrence is its first occurrence in $\mathbf c$. By the claim, for fixed $\ell$, all but at most $4\ell(\ell+1)$
of $M_D(a)$ are good. 

\begin{claim}\label{claim:existence-good-subword}
Every sufficiently long subword of $W_D$ contains a good $M_D(a)$.
\end{claim}
\begin{proof}[Proof of the claim]
    The natural occurrence of $M_D(a)$ (resp. $M_D(a+1)$) starts at the integer $a10^{D-\ell}-1 $ (resp. $(a+1)10^{D-\ell}-1$). These two integers differ by $10^{D-\ell}$. Since every integer in $W_D$ has exactly $D$ digits, the starting positions of two consecutive $M_D(a)$ are separated by
$D10^{D-\ell}$ digits. 

Let $\mathbf c_0\subset W_D
$ be a finite subword satisfying
$|\mathbf c_0|\ge\gamma D10^{D-1}.$
Since we may lose only at most 2 finite subwords of the form $M_D(a)$ near the two endpoints of $\mathbf c_0$, the number of $M_D(a)$ contained in $\mathbf c_0$ is at least
$$
\frac{|\mathbf c_0|}{D10^{D-\ell}}-2.
$$
Therefore
$$
|\{a:M_D(a)\subset\mathbf c_0\}|
\ge
\gamma10^{\ell-1}-2.
$$

We now choose $\ell$ sufficiently large that $\gamma10^{\ell-1}-2>4\ell(\ell+1).$ It follows that $\mathbf c_0$ contains more than $4\ell(\ell+1)$ the finite words $M_D(a)$, while by the claim when $D>2\ell$, at most $4\ell(\ell+1)$ of  $M_D(a)$ are bad. Hence $\mathbf c_0$ contains at least one good $M_D(a)$ as desired. 
\end{proof}

We apply Claim \ref{claim:existence-good-subword} to the subword $\mathbf c'_n\subset W_{D_n}$ constructed above. Recall that $|\mathbf c'_n|\ge\gamma D_n10^{D_n-1}.$
Hence for $n$ sufficiently large, $\mathbf c'_n$ contains a good $M_{D_n}(a_n)$ at its natural
occurrence. Let $q_n$ be the starting position of this occurrence. Since
$\mathbf c'_n\subset \mathbf c[s_n+1,s_n+\lfloor\eta s_n\rfloor],$
and $\mathbf c[s_n+1,s_n+\lfloor\eta s_n\rfloor]=\mathbf c[r_n+1,r_n+\lfloor\eta s_n\rfloor],$ the word $M_{D_n}(a_n)$ occurs at the position $q_n-(s_n-r_n)<q_n.$ Thus $M_{D_n}(a_n)$ has an occurrence before its natural occurrence, contradicting the fact that $M_{D_n}(a_n)$ is good. We complete the proof.
\end{proof}

\section{The Rudin--Shapiro and Thue--Morse words}\label{section: Thue--Morse}
In this section, we give an upper bound for the refined Diophantine exponent of the Rudin--Shapiro and Thue--Morse words. We focus first on the Thue--Morse word. The approach for the Rudin--Shapiro word is similar.

Recall that the Thue--Morse word $\mathbf{T}$\footnote{We use $\mathbf{T}$ instead of the usual notation $\mathbf{t}$ in order to avoid confusion with the sequence $(t_n)_{n\geq0}$ in the definition of $\mathbf{Rdio}$.} is the unique fixed point of the morphism $\varphi$ on the free semigroup generated by $\{0,1\}$ defined by $\varphi(0) = 01$ and $\varphi(1) = 10$.  

In the following proof, we use an induction argument involving carries in base 2 to establish a lower bound, linear in $N$, for the quantity $D(\mathbf{T};m,k,N)$. The general idea is that if there exist constants $ C_1,C_2 > 0$ such that\[D(\mathbf{T};m,k,N) \geq C_1 N - C_2 k,\]then one can easily bound $\mathbf{Rdio}$ from above by $1 + {C_2}/{C_1}$. To this end, we follow the computational strategy in \cite{Mauduit-Sarkozy-1998-II}. Note that this condition is weaker than the one in Remark~\ref{remark: lower bound for Dbmkn}, so the construction in the proof of Theorem~\ref{Theorem: spectrum of Rdio} does not work with the Thue--Morse word. 

\begin{theorem}\label{theorem: Thue Morse}
The refined Diophantine exponent of the Thue--Morse sequence is at most $25$.
\end{theorem}
\begin{proof}
For convenience, we replace $\{0, 1\}$ by $\{1, -1\}$, so $T_n = (-1)^{s_2(n)}$, where $s_2(n)$ is the sum of the binary digits of $n$. It follows that $T_{2i} = T_i$ and $T_{2i+1} = -T_i$. Recall from Definition~\ref{definition: Cbmkn Dbmkn} that for $m\geq0$, $k\geq1$ and $N\geq1$, the number of mismatches between $\mathbf{T}[m,m+N-1]$ and $\mathbf{T}[m+k,m+N-1+k]$  is

\[D(\mathbf{T};m,k,N) = \sum_{i=m}^{m+N-1} \frac{1 - T_i T_{i+k}}{2}.\]
We denote   \[e_i(k)=\frac{1 - T_i T_{i+k}}{2}.\] We have the following recurrence relations $e_{2i}(2k) = e_i(k)$ and $e_{2i+1}(2k) = e_i(k)$. Thus\[D(\mathbf{T};0, 2k, 2N) = 2D(\mathbf{T};0, k, N).\]
Similarly, from $e_{2i}(2k+1) = 1 - e_i(k)$ and $e_{2i+1}(2k+1) = 1 - e_i(k+1)$, we deduce that 
    \[D(\mathbf{T};0, 2k+1, 2N) = 2N - D(\mathbf{T};0, k, N) - D(\mathbf{T};0, k+1, N).\]

\begin{claim}\label{claim: bounds for D0k2p}
   For $k\geq1$ and $p\geq0$, we have \[\frac{1}{3}2^p - k \leq D(\mathbf{T};0, k, 2^p) \leq \frac{2}{3}2^p + k.\]
\end{claim}
\begin{proof}[Proof of the claim]
First, we treat the case $k=1$ separately. We observe that $D(\mathbf{T};0,1,2^{p+1})=2^{p+1}-D(\mathbf{T};0,1,2^{p})$ and $D(\mathbf{T};0,1,1)=1$. Thus $D(\mathbf{T};0,1,2^{p})=\dfrac{2^{p+1}+(-1)^p}{3}$ which satisfies the desired inequalities.

Now,  we prove the claim by induction on $p$. Since $D(\mathbf{T};0, k, 1) \in \{0, 1\}$, the base case $p=0$ holds.

  Assume the bounds hold for $p\geq0$. Using the even recurrence we have\[D(\mathbf{T};0, 2k, 2^{p+1}) = 2D(\mathbf{T};0, k, 2^p) \geq 2\left(\frac{1}{3}2^p - k\right) = \frac{1}{3}2^{p+1} - 2k\]
  and \[D(\mathbf{T};0, 2k, 2^{p+1}) = 2D(\mathbf{T};0, k, 2^p) \leq 2\left(\frac{2}{3}2^p + k\right) = \frac{2}{3}2^{p+1} + 2k.\]
When $2k+1\geq3$, we have \[D(\mathbf{T};0, 2k+1, 2^{p+1}) = 2^{p+1} - D(\mathbf{T};0, k, 2^p) - D(\mathbf{T};0, k+1, 2^p).\]
It follows that \[D(\mathbf{T};0, 2k+1, 2^{p+1}) \geq 2^{p+1} - \left(\frac{2}{3}2^p + k\right) - \left(\frac{2}{3}2^p + k + 1\right) = \frac{1}{3}2^{p+1} - (2k+1)\]
and \[D(\mathbf{T};0, 2k+1, 2^{p+1}) \leq 2^{p+1} - \left(\frac{1}{3}2^p - k\right) - \left(\frac{1}{3}2^p - (k+1)\right)=\frac{2}{3}2^{p+1} + (2k+1)\]as desired.
\end{proof}

\begin{claim}\label{claim: bound for DmkN}
 For  $m\geq0$, $k\geq1$ and $N\geq1$, we have \[D(\mathbf{T};m, k, N) > \frac{1}{12}N - 2k.\]
\end{claim}
\begin{proof}[Proof of the claim]
If $N \leq 24k$, then $N/12 - 2k \leq 0$, making the bound trivially true. Thus, we may and do assume $N > 24k$.

    Let $p \geq 0$ be the unique integer such that $2^{p+1} \leq N < 2^{p+2}$. Thus the interval $\mathbf{T}[m, m+N-1]$ contains the finite subword  $[c 2^{p}, c 2^{p} + 2^{p} - 1]$ where $c = \lceil {m}/{2^p} \rceil$. We have
    \begin{equation} \label{eq:lower bound for DmkN}D(\mathbf{T};m, k, N) = \sum_{i=m}^{m+N-1} e_i(k) \geq \sum_{j=0}^{2^{p}-1} e_{c 2^{p} + j}(k)\end{equation}For any $x < 2^{p}$, we have 
    \[s_2(c 2^{p} + x) = s_2(c) + s_2(x),\text{ i.e., } T_{c 2^{p} + x} = T_c T_x.\]
   In particular, for $0 \leq j < 2^p - k$ (note that $2^p>k$), we have 
   \begin{equation*}T_{c 2^p + j} = T_c T_j, \quad T_{c 2^p + j + k} = T_c T_{j+k}\end{equation*} 
   It follows that\begin{align*}
    e_{c 2^p + j}(k) = \frac{1 - (T_c T_j)(T_c T_{j+k})}{2} &= \frac{1 - T_c^2 T_j T_{j+k}}{2}\\   &= \frac{1 - T_j T_{j+k}}{2} = e_j(k)
   \end{align*} 
   Therefore \begin{equation}\label{eq: lower bound for sum e_2p-1} \sum_{j=0}^{2^p-1} e_{c 2^p + j}(k) \geq \sum_{j=0}^{2^p-k-1} e_{c 2^p + j}(k) = \sum_{j=0}^{2^p-k-1} e_j(k)\end{equation}
   On the other hand, we have
   \begin{equation}\label{eq: lower bound for sum e_2p-k-1} 
       \sum_{j=0}^{2^p-k-1} e_j(k) = D(\mathbf{T};0, k, 2^p) - \sum_{j=2^p-k}^{2^p-1} e_j(k)\geq D(\mathbf{T};0, k, 2^p) - k
   \end{equation} 
   Using Claim~\ref{claim: bounds for D0k2p} and \eqref{eq:lower bound for DmkN}, \eqref{eq: lower bound for sum e_2p-1}, \eqref{eq: lower bound for sum e_2p-k-1}, we deduce that \[D(\mathbf{T};m, k, N) \geq D(\mathbf{T};0, k, 2^p)-k \geq \left( \frac{1}{3}2^p - k \right) - k = \frac{1}{3}2^p - 2k.\] Since $N<2^{p+2}$, we conclude that \[D(\mathbf{T};m, k, N) > \frac{1}{12}N - 2k.\]
\end{proof}

Now let $\rho<\mathbf{Rdio}(\mathbf{T})$, then $\mathbf{T}$ satisfies Condition $(*)_\rho$. Let $\epsilon \in (0, {1}/{12})$, then there exist sequences $(r_n)$, $(s_n)$, $(t_n)$ and $\delta\geq0$  such that $t_n \geq \rho s_n$ and \[D(\mathbf{T};r_n+1, s_n-r_n, t_n-s_n) \leq \epsilon (t_n-s_n).\] By Claim~\ref{claim: bound for DmkN}, we have
\begin{align*}
    \frac{1}{12}(t_n-s_n) - 2(s_n-r_n) < D(\mathbf{T};r_n+1, s_n-r_n, t_n-s_n) \leq \epsilon (t_n-s_n), 
\end{align*}
so \[\left(\frac{1}{12} - \epsilon\right) (t_n-s_n) <  2(s_n-r_n).\]
This yields
\[ \left(\frac{1}{12} - \epsilon\right)(\rho - 1)(s_n-r_n-1) < 2(s_n-r_n). \]
Let $n$ tend to infinity, we obtain
\[ \left(\frac{1}{12} - \epsilon\right)(\rho - 1) \leq 2. \]
It follows that
\[ \rho \leq 1 + \frac{2}{\frac{1}{12} - \epsilon}. \]
Let $\epsilon$ tend to $0$, we have $\rho\leq25$ as desired. 
\end{proof}

\begin{remark}
    Note that with $m=0$ and $k=1$, the proof in fact yields that\[C(\mathbf{T}; 0, 1, N)  \asymp \frac{N}{3},\]
    so the Thue--Morse word $\mathbf{T}$ does not satisfy the condition of Theorem~\ref{theorem: existence of word}. That partly explains the strength of this theorem.
\end{remark}

Next, we turn to the Rudin--Shapiro word. 
Again for convenience, we consider the alphabet $\{1, -1\}$. The Rudin--Shapiro sequence $\mathbf{R} = R_0 R_1 R_2 \cdots$ over $\{1, -1\}$ is defined by $R_n = (-1)^{r(n)}$, where $r(n)$ is the number of occurrences of the block $11$ in the binary expansion of $n$. In this situation, we have $R_{2i} = R_i$ and $R_{2i+1} = (-1)^i R_i$. Applying the strategy used in the proof of Theorem~\ref{theorem: Thue Morse}, we obtain the following bound. 

\begin{theorem}\label{theorem: Rudin Shapiro}
The refined Diophantine exponent of the Rudin--Shapiro sequence is at most $25$.
\end{theorem}
\begin{proof}
To keep the paper reasonably short, we only sketch the proof.
As in the proof of Theorem~\ref{theorem: Thue Morse}, for $m \geq 0$, $k \geq 1$ and $N \geq 1$, the number of mismatches between $\mathbf{R}[m,m+N-1]$ and $\mathbf{R}[m+k,m+N-1+k]$ is
\[D(\mathbf{R};m,k,N) = \sum_{i=m}^{m+N-1} \frac{1 - R_i R_{i+k}}{2}.\]
To analyze $D(\mathbf{R};m,k,N)$, we define 
\[ U(k, p) = \sum_{i=0}^{2^p-1} R_i R_{i+k}, \quad V(k, p) = \sum_{i=0}^{2^p-1} (-1)^i R_i R_{i+k}. \]

\begin{claim}\label{claim: RS correlation bound}
    For $k \geq 1$ and $p \geq 0$, we have $|U(k, p)| \leq 2k$ and $|V(k, p)| \leq 2k$.
\end{claim}
\begin{proof}[Proof of the claim]
This follows from induction on $p$. We omit the details.
\end{proof}

As a consequence, noting that $D(\mathbf{R};0, k, 2^p) = 2^{p-1} - U(k,p)/2$, we obtain a key inequality $D(\mathbf{R};0, k, 2^p) \geq 2^{p-1} - k$.

\begin{claim}\label{claim: RS DmkN bound}
For $m \geq 0$, $k \geq 1$ and $N \geq 1$, we have \[D(\mathbf{R};m, k, N) > \frac{1}{12}N - 2k.\]
\end{claim}
\begin{proof}[Proof of the claim]
This can be proved similarly to Claim~\ref{claim: bound for DmkN}. We omit the details.
\end{proof}

As at the end of the proof of Theorem~\ref{theorem: Thue Morse}, we obtain $\mathbf{Rdio}(\mathbf{R}) \leq 25$. 
\end{proof}

\begin{remark}\label{remark: Champernowne Rudin--Shapiro}
One the other hand, it is known that $\mathbf{Dio}(\mathbf{R})\leq 4$, as the critical exponent of $\mathbf{R}$ is bounded above by 4; see e.g. \cite{Allouche-Melou-1994}.
\end{remark}

\section{Coding rotations by intervals}\label{section: coding rotations}

In this section, we show that the refined Diophantine exponent of a word arising from the \emph{coding of an irrational rotation by intervals} is always infinite, unless in the trivial case.

Let $\theta\in[0,1)$ be irrational, and let $x \in \mathbb{R}$. Let $\Sigma$ be a finite alphabet. Let $c\colon [0, 1) \to \Sigma$ be a \emph{piecewise-constant function} mapping to $\Sigma$, with \emph{partition boundaries} $0=\gamma_0<\gamma_1<\cdots<\gamma_d=1$. In this section, we consider the infinite word $\mathbf{a}=a_0a_1\cdots$ defined by $a_i = c(\{i\theta + x\})$, which represents the coding of the rotation $\theta$ by intervals $[\gamma_{j-1},\gamma_j)$. To establish meaningful dynamical properties, we assume that $c$ is non-constant. 

\begin{remark}
    By mimicking the argument in \cite[Remark~2.22]{Khai-2026}, one can show that the word $\mathbf{a}$ has sublinear complexity. It is also clear that $\mathbf{a}$ is aperiodic.
\end{remark}

\begin{theorem}\label{Theorem: coding rotations}
For such an infinite word $\mathbf{a}$, we always have $\mathbf{Rdio}(\mathbf{a})=\infty$.
\end{theorem}

\begin{proof}
We fix any $\rho > 1$. We will construct sequences $(r_n)$, $(s_n)$, $(t_n)$ and an integer $\delta$ independent of $n$. Let $({p_n}/{q_n})_{n \geq 0}$ be the continued fraction convergents of $\theta$.  We set $r_n = -1$, $s_n = q_n - 1$, $t_n = q_n - 1 + \lceil \rho q_n \rceil$, and $\delta = d\lceil \rho + 1 \rceil$. It is clear that $t_n \geq \rho s_n$. 

If an index $i$ is a mismatch with respect to $s_n-r_n=q_n$ (meaning $a_i \neq a_{i+q_n}$), then the points $\{i\theta + x\}$ and $\{i\theta + x + q_n\theta\}$ must fall into different intervals of the partition. Consequently, the distance from $\{i\theta + x\}$ to some $\gamma_j$ must be at most $\|q_n\theta\| = |q_n\theta - p_n|$. For each $1 \leq j \leq d$, let $I_j^{(n)} \subset [0, 1)$ be the interval of length $|q_n\theta - p_n|$ adjacent to $\gamma_j$. The necessary condition for a mismatch requires that $\{i\theta + x\} \in \bigcup_{j=1}^d I_j^{(n)}$.

\begin{claim}\label{claim: mismatches coding rotation}
    The exact number of mismatches in any subword of length $q_n$ is bounded above by $d$. That is, for any integer $m \geq 0$, we have
    \[|\{m\leq i < m+q_n : a_{i}\neq a_{i+q_n}\}| \leq d.\]
\end{claim}
\begin{proof}[Proof of the claim]
By the law of best approximation, the minimum distance between any two points in an orbit of length $q_n$ is
\[\min_{0 < |i-j| < q_n} \|(i\theta + x) - (j\theta + x)\| = \min_{0 < |i-j| < q_n} \|(i-j)\theta\| = |q_{n-1}\theta-p_{n-1}|.\]
Because $|q_n\theta - p_n| < |q_{n-1}\theta-p_{n-1}|$, we deduce that the intersection
\[ \{ {i\theta + x} : m \leq i < m+q_n \} \cap \bigcup_{j=1}^d I_j^{(n)}\]
contains at most $1$ point per interval $I_j^{(n)}$. Thus, the desired number of mismatches is bounded above by $d$.
\end{proof}

Now, by covering the block $[0, t_n-s_n-1]$ with $\left\lceil ({t_n-s_n})/{q_n} \right\rceil$ adjacent intervals of length at most  $q_n$ and applying the claim, the total number of mismatches is at most 
\[ d\left\lceil \frac{t_n-s_n}{q_n} \right\rceil  = d\left\lceil \frac{\lceil \rho q_n \rceil}{q_n} \right\rceil  \leq d\lceil \rho + 1 \rceil = \delta. \]
It follows that condition $(*)_\rho$ holds for any $\rho>1$, so $\mathbf{Rdio}(\mathbf{a}) = \infty$.
\end{proof}

To conclude the section, we determine the Diophantine exponent of words coming from the coding of a rotation by intervals.

\begin{theorem}\label{theorem: Dio coding rotation}
   Let $\mathbf{a}=a_0a_1\cdots$ be the word considered above. We have the following dichotomy.
   \begin{itemize}
       \item If $\theta$ is badly approximable, then $ 1 < \mathbf{Dio}(\mathbf{a}) < \infty$.
       \item If $\theta$ is well approximable, then $ \mathbf{Dio}(\mathbf{a}) = \infty$.
   \end{itemize}
\end{theorem}
\begin{proof}
\begin{itemize}
    \item When $\theta$ is badly approximable, since $\mathbf{a}$ has sublinear complexity, the lower bound $  \mathbf{Dio}(\mathbf{a})>1$ follows from \cite[Proposition~9.1]{Adamczewski-Bugeaud-2011}. 

For the finiteness of $\mathbf{Dio}(\mathbf{a})$,  we first reduce to the situation where $c$ has no non-trivial rotational symmetry. Let
\[
H(c):=
\{\alpha\in \mathbb R/\mathbb Z:
c(y+\alpha)=c(y)\text{ for almost every }y\}
\] with respect to the Lebesgue measure $\mu$. We claim that $H(c)$ is a finite subgroup of $\mathbb R/\mathbb Z$. Indeed, if $\alpha\in H(c)$, then
$c(\cdot+\alpha)$ and $c$ have the same discontinuities, and hence
\[
\{\gamma_j:0\leq j\leq d-1\}+\alpha=\{\gamma_j:0\leq j\leq d-1\}.
\]
Thus $H(c)$ acts on the finite set $\{\gamma_j:0\leq j\leq d-1\}$. We note that this action is faithful, hence $H(c)$ is finite. It follows that there exists an integer $m\ge1$ such that
\[
H(c)=
\left\{
0,\frac1m,\ldots,\frac{m-1}{m}
\right\}.
\]
We then define $\widetilde c(z):=c\left({z}/{m}\right), z\in[0,1).$ Since $j/m\in H(c)$ for every $0\le j<m$, we have $c(y)=\widetilde c(\{my\})$ for almost every $y\in\mathbb R/\mathbb Z.$ Indeed, we write
\[
my=j+z,
0\le j<m,
z=\{my\},
\]
then $y={j}/{m}+{z}/{m}$ and therefore, for almost every $y$,
\[
c(y)
=
c\left(\frac{z}{m}+\frac{j}{m}\right)
=
c\left(\frac{z}{m}\right)
=
\widetilde c(z).
\]
Further, we set $\widetilde\theta=\{m\theta\},$ and $\widetilde x=\{mx\}.$
Then
\[
\widetilde c(\{\widetilde x+n\widetilde\theta\})
=
c(\{x+n\theta\})
\]
for all but finitely many $n$ since the two functions $c(y)$ and $\widetilde c(\{my\})$
agree almost everywhere and hence differ only at finitely many boundary
points as $\theta$ is irrational.
Thus the two coding words are eventually equal and consequently have the
same Diophantine exponent by \cite[Proposition 6.4]{Peltomaki-2024}. 

Moreover, $\widetilde c$ has no non-zero rotational symmetry. Indeed, if
$\beta\in H(\widetilde c)$, we choose its representative with
$0\le\beta<1$ and set $\alpha={\beta}/{m}.$ Then
\[
c(y+\alpha)
=
\widetilde c(\{my+\beta\})
=
\widetilde c(\{my\})
=
c(y)
\]
for almost every $y$, so $\alpha\in H(c)$. Since $0\le\alpha<\frac1m$
and the elements of $H(c)$ are $j/m$, we have
$\alpha=0$, hence $\beta=0$. Therefore $H(\widetilde c)=\{0\}.$

In addition, the bad approximability is preserved under multiplication by $m$. Indeed, if there is $\kappa>$ such that $\|r\theta\|\ge{\kappa}/{r}$ for $r\geq1$ 
then
\[
\|q\widetilde\theta\|
=
\|qm\theta\|
\ge
\frac{\kappa}{qm}
=
\frac{\kappa/m}{q}.
\]
Hence $\widetilde\theta$ is also badly approximable.

Therefore, replacing $(c,\theta,x)$ by
$(\widetilde c,\widetilde\theta,\widetilde x)$, we may and do assume without loss of generality that $H(c)=\{0\}.$

Now, we follow the idea of \cite{Adamczewski-Bugeaud-2011,Mignosi-1989} to show that $\mathbf{Dio}(\mathbf{a})<\infty$. Namely, $\mathbf{Dio}(\mathbf{a})$ is bounded above by the \emph{index} of $\mathbf{a}$, which is defined as the supremum of $w\geq1$ such that there exists a non-empty finite word $V$  such that $V^w$ is a finite subword of $\mathbf{a}$. It remains to show the finiteness of such an index. 

To this end, assume the contradiction that the index of $\mathbf{a}$ is infinite. Let $w>1$ be a large real number to be chosen later, and $V^w$ be a finite subword starting at index $k \geq 0$ with period $q = |V| \geq 1$. Thus for $k\leq i\leq k+\lfloor (w-1)q  \rfloor-1$, we have $a_i = a_{i+q}$, i.e., $c(\{i\theta + x\}) = c(\{i\theta+q\theta + x\})$.

To detect the mismatches, we need the quantity \[\Delta = \min_{1 \leq j \leq d} (\gamma_j - \gamma_{j-1}) > 0.\] We have two cases.

\textbf{Case 1}: $\|q\theta\|\geq \Delta$. Then $\{q\theta\}\in[\Delta,1-\Delta].$
Let $E_q=\{y\in\mathbb R/\mathbb Z:c(y)\neq c(y+q\theta)\},$
and define $f(\alpha)=\mu\{y\in\mathbb R/\mathbb Z:c(y)\neq c(y+\alpha)\}$. Then $f$ is continuous.  Moreover, $f(\alpha)=0$
if and only if $c(y+\alpha)=c(y)$
for almost every $y$. Since $H(c)=\{0\}$, it follows that
\[
f(\alpha)>0
\text{ for every }\alpha\neq0.
\]
Therefore, by compactness,
\[
\nu
:=
\min_{\alpha\in[\Delta,1-\Delta]}f(\alpha)
>0.
\]
It follows that $\mu(E_q)=f(\{q\theta\})\geq\nu.$ Note that the boundary points of $E_q$ belong to
\[
\{\gamma_j:0\leq j\leq d\}\cup(\gamma_j-q\theta:0\leq j\leq d),
\]
so $E_q$ is a union of at most $2d$ intervals. Hence one of its interval
components, say $I$, satisfies
\[
|I|
\geq
\frac{\mu(E_q)}{2d}
\geq
\frac{\nu}{2d}.
\]
We then set $\eta:=\frac{\nu}{2d}>0$ and $L:=\lfloor(w-1)q\rfloor.$ Since $V^w$ has period $q$, we have
$a_i=a_{i+q}$ for $k\leq i<k+L$.
Thus the $L$ orbit points
\[
\{x+k\theta\},\{x+(k+1)\theta\},\ldots,
\{x+(k+L-1)\theta\}
\]
avoid $E_q$, and hence avoid the interval $I$.

Since $\theta$ is irrational, the orbit $\{j\theta:j\geq0\}$ is dense in
$\mathbb R/\mathbb Z$. Hence there exists $M=M(\theta,\eta)$ depending on $\theta$ and $\eta$ such that
\[
\{0,\theta,\ldots,(M-1)\theta\}
\]
is $\eta/3$-dense in $\mathbb R/\mathbb Z$. It follows that every set of $M$ consecutive orbit points meets every interval of length at least
$\eta$.

Since our $L$ consecutive orbit points avoid $I$, whose length is at least
$\eta$, we  have $L<M.$
Thus
\[
\lfloor(w-1)q\rfloor<M,
\]
so
\[
w<1+\frac{M}{q}\leq M+1.
\]
This contradicts the choice of $w$ if we choose $w>M+1$ at the beginning.

\textbf{Case 2}: $\|q\theta\|<\Delta$. We set $L:=\lfloor (w-1)q\rfloor.$ Since $\|q\theta\|<\Delta$, there exists an interval $I\subset [0,1)$ of length
$\|q\theta\|$ such that
\[
c(y)\neq c(y+q\theta)
\text{ for every }y\in I.
\]
Since $V^w$ has period $q$, the $L$ points $\{i\theta+x\},
k\le i<k+L,$ avoid $I$.

Let $q_n$ be the largest continued fraction denominator of $\theta$ such
that $q_n\le L.$ Thus $q_n\le L<q_{n+1}.$ Here, if $L<q_1$, then we obtain a bound for $w$ and we are done.

We claim that the largest gap between any $q_n$ consecutive points of the orbit of $\theta$ is less than
\[
\frac1{q_n}+\frac2{q_{n+1}}<\frac3{q_n}.
\]
Indeed, for $0\le j<q_n$, we have
\[
\left\|j\theta-\frac{jp_n}{q_n}\right\|
\le
j\left|\theta-\frac{p_n}{q_n}\right|
<
\frac1{q_{n+1}}.
\]
Hence every point of the circle lies within the circle of radius $
\frac1{2q_n}+\frac1{q_{n+1}}$ centered at one of the first $q_n$ orbit points of $\theta$, so their circular distances are at most
\[
\frac1{q_n}+\frac2{q_{n+1}}
<
\frac3{q_n}.
\]

As these $L\geq q_n$ points $\{i\theta+x\},
k\le i<k+L,$ avoid the interval $I$ of length $\|q\theta\|$, we obtain
\[
\|q\theta\|<\frac3{q_n}.
\]

Recall that $\theta$ is badly approximable. Hence there exist constants
$A,\kappa>0$ such that $q_{n+1}\le Aq_n$ and $\|r\theta\|\ge\frac{\kappa}{r}$ for every $r\ge1.$ In particular,
\[
\frac{\kappa}{q}
\le
\|q\theta\|
<
\frac3{q_n}.
\]
Moreover, since $L<q_{n+1}\le Aq_n,$
we have
\[
\frac1{q_n}<\frac{A}{L}.
\]
Therefore $\frac{\kappa}{q}
<
\frac{3A}{L},$
and hence
\[
L<\frac{3A}{\kappa}q.
\]
Since $L=\lfloor(w-1)q\rfloor$, we obtain
\[
(w-1)q<L+1
<
\frac{3A}{\kappa}q+1.
\]
Thus
\[
w
<
1+\frac{3A}{\kappa}+\frac1q
\le
2+\frac{3A}{\kappa}.
\]
This contradicts the choice of $w$ if we choose at the beginning that $w\ge 2+\frac{3A}{\kappa}.$

  \item When $\theta$ is well approximable, let $({p_n}/{q_n})_{n\geq1}$ be the continued fraction convergents of $\theta$. We will use $q_n$ as periods to construct data for $\mathbf{Dio}(\mathbf{a})$ as in the proof of Theorem~\ref{Theorem: coding rotations}.  For now, we fix $n$. Similar to Claim~\ref{claim: mismatches coding rotation}, we have the following.

\begin{claim}\label{claim: mismatches well approximable}
    The number of mismatches in any subword of length $q_{n+1}$ is bounded above by $2d$. That is, for any integer $m \geq 0$, we have
    \[|\{m\leq i < m+q_{n+1} : a_{i}\neq a_{i+q_n}\}| \leq 2d.\]
\end{claim}

It follows that there are $0 = i_0 \leq i_1 < i_2 < \dots < i_r < i_{r+1} = q_{n+1}$, where $0 \leq r \leq 2d$, such that the mismatches $a_i \neq a_{i+q_n}$ within the interval $[1, q_{n+1}-1]$ occur at the indices $i \in \{i_1, \dots, i_r\}$.
  
For each $0 \leq j \leq r$, we consider the prefix $U_j = \mathbf{a}[0, i_j]$ and the periodic finite subword $V_j = \mathbf{a}[i_j+1, i_j+q_n]$. Then $U_j V_j^w$ is a prefix of $\mathbf{a}$ with suitable $w \geq 1$ such that $|U_j V_j^w| = i_{j+1}+q_n$. We have 
\[ \frac{|U_j V_j^w|}{|U_j V_j|} = \frac{i_{j+1}+q_n}{i_j + 1 + q_n}. \]
We set $x_j:=i_j+1+q_n$ and \[M_n := \max_{0 \leq j \leq r} \frac{|U_j V_j^w|}{|U_j V_j|}.\]
Then $x_{j+1}-1\le M_nx_j.$ Since $M_n\ge1$ and $x_j\ge1$, we obtain
\[
x_{j+1}\le 2M_nx_j.
\]
As $x_0=q_n+1,$  induction gives us
\[
x_j\le (2M_n)^j(q_n+1),
\text{ for } 0\le j\le r+1.
\]
Therefore
\[
q_{n+1}+q_n+1
\le
(2M_n)^{r+1}(q_n+1),
\]
and hence
\[
2M_n
\ge
\left(
\frac{q_{n+1}+q_n+1}{q_n+1}
\right)^{1/(r+1)}
\ge
\left(
\frac{q_{n+1}+q_n+1}{q_n+1}
\right)^{1/(2d+1)}.
\]
Hence $(M_n)$ is unbounded. Thus there exists a subsequence $(n_k)$ such that $M_{n_k}\to\infty.$ For each $k$, we choose $j_k$ maximizing $M_{n_k}$.
Then we have $|V_{j_k}|=q_{n_k}\to\infty$
and $\frac{|U_{j_k}V_{j_k}^{w_{j_k}}|}
{|U_{j_k}V_{j_k}|}
=M_{n_k}\to\infty.$ By Remark \ref{remark:Dio}, we obtain $\mathbf{Dio}(\mathbf{a}) = \infty$ as desired.
\end{itemize}    
\end{proof}

\section{Bracket words}\label{section: bracket}
We study bracket words in this section. We start by recalling this notion, following \cite{Adamczewski-Konieczny-2023}.

\begin{definition}
The family of \emph{generalized polynomials} is the smallest set of functions $\mathbb{Z} \to \mathbb{R}$ containing the polynomial maps and closed under addition, multiplication, and the operation of taking the integer part $\lfloor \cdot \rfloor$. 
\end{definition}

\begin{definition}
A \emph{bracket word} over a finite alphabet $\Sigma$ is an infinite word $\mathbf{a} = a_0a_1\cdots$ of the form $a_n = a(g(n))$, where $g\colon \mathbb{N}_0 \to \mathbb{R}$ is a finitely-valued generalized polynomial map and $a\colon g(\mathbb{N}_0) \to \Sigma$ is an arbitrary map.
\end{definition}


\begin{example}
\begin{enumerate}
    \item The simplest non-trivial family of generalized polynomials includes those that directly correspond to the codings of a rotation by intervals. 

Let $\theta\in[0,1)$ and $x\in\mathbb{R}$. We partition the unit interval $[0, 1)$ into $d \geq 2$ sub-intervals using boundary points $0 = \gamma_0 < \gamma_1 < \dots < \gamma_d = 1$. Let $g(i) = \{i\theta+x\}$. The characteristic function indicating whether $g(i)$ falls into the interval $[\gamma_{j-1}, \gamma_j)$ can be written using the floor function as\[\lfloor g(i) + 1 - \gamma_{j-1} \rfloor - \lfloor g(i) + 1 - \gamma_j \rfloor.\] Note that the polynomial $i\theta+x$ is of degree 1 in $i$. 
    \item More generally, let $g(x) = \sum_{s=0}^\ell b_s x^s \in \mathbb{R}[x]$ be a polynomial of degree $\ell \geq 1$. Let $\mathbf{a}$ be defined by $a_n = c(\{g(n)\})$, where $c\colon[0,1)\to\Sigma$ is a non-constant piecewise-constant function with \emph{partition boundaries} $0=\gamma_0<\gamma_1<\cdots<\gamma_d=1$ at which $c$ changes values. Then $\mathbf{a}$ is a bracket word.
\end{enumerate}

\end{example}

 Motivated by \cite{Mauduit-1999-IV,Mauduit-2000-V}, the goals now are to establish sufficient conditions for higher-degree bracket words that force $\mathbf{Rdio} < \infty$  (resp. $\mathbf{Rdio} = \infty$).

\begin{theorem}\label{theorem: bracket finite Rdio}
Let $g(x) = \sum_{s=0}^\ell b_s x^s \in \mathbb{R}[x]$ be a polynomial of degree $\ell \geq 2$. Let $\mathbf{a}$ be defined by $a_n = c(\{g(n)\})$, where $c\colon[0,1)\to\Sigma$ is a non-constant piecewise-constant function with partition boundaries $0=\gamma_0<\gamma_1<\cdots<\gamma_d=1$. Assume the leading coefficient $b_\ell$ is badly approximable. Then $\mathbf{Rdio}(\mathbf{a}) < \infty$.
\end{theorem}

\begin{remark}
  Before giving the proof, note that when $\ell=1$, the shift $(i+k)\theta - i\theta = k\theta \pmod 1$ is independent of $i$. This observation leads to the proof of Theorem~\ref{Theorem: coding rotations} based on the equidistribution of $(i\theta+x)_{i\geq0}$ to bound the discrepancy. When $\ell\geq 2$, the shift becomes dependent on $i$, so the above equidistribution result is no longer sufficient to bound the discrepancy. To this end, we follow \cite{Mauduit-2000-V} by applying the Erd\H{o}s--Tur\'an--Koksma inequality.
\end{remark}

\begin{proof}
    Assume for contradiction that $\mathbf{Rdio}(\mathbf{a}) = \infty$. Let \[\Delta = \min_{1 \leq j \leq d} (\gamma_j - \gamma_{j-1}) > 0.\] Let $\rho > 2$ (to be chosen later), and $\epsilon$ such that $0 < \epsilon < {\Delta^2}/{9}$. Then there exist  sequences $(r_n), (s_n), (t_n)$ and $\delta\geq0$ satisfying conditions in $(*)_\rho$.

For $1 \leq i \leq t_n-s_n$, define $\vec{v}_i^{(n)} = (x_i, y_i) \in [0,1)^2$ as follows
\begin{align*}
x_i &= g(r_n + i) \pmod 1 \\
y_i &= g(r_n + i + s_n - r_n) - g(r_n + i) \pmod 1
\end{align*}
We denote by $D_n$ the two-dimensional discrepancy with respect to $\vec{v}_i^{(n)}$, that is
\[D_n = \sup_{B \subset [0,1)^2} \left\vert{} \frac{|\{1\leq i\leq t_n-s_n:\vec{v}_i^{(n)}\in B\}|}{t_n - s_n} - \mu(B) \right\vert{}\] where $B$ runs over sub-rectangular of $[0,1)^2$ and $\mu$ denotes the Lebesgue measure. 

To give an upper bound for $D_n$, we follow the method used to prove \cite[Theorem~3]{Mauduit-2000-V}, where Mauduit and S\`ark\"{o}zy use the Erd\H{o}s--Tur\'an inequality \cite{Erdos-Turan-1948} to bound the one-dimensional discrepancy. In our case, we need to use the two-dimensional version of this inequality. By the Erd\H{o}s--Tur\'an--Koksma inequality (see e.g. \cite[Theorem~1.21]{Drmota-Tichy1997}), we have for any integer $H \geq 1$ that
\begin{equation}\label{eq: ETK}
    D_n \leq \frac{9}{4} \left( \frac{2}{H+1} + \sum_{0 < \|(h_1, h_2)\|_\infty \leq H} \frac{1}{r(h_1, h_2)} \left| \frac{1}{t_n-s_n} \sum_{i=1}^{t_n-s_n} e^{2\pi \mathrm{i} (h_1 x_i + h_2 y_i)} \right| \right)  
\end{equation}
    where $\|(h_1, h_2)\|_\infty=\max(|h_1|,|h_2|)$ and $r(h_1, h_2) = \max(1, |h_1|) \max(1, |h_2|)$.

To simplify, we denote \[S_n(h_1, h_2) = \frac{1}{t_n-s_n} \sum_{i=1}^{t_n-s_n} e^{2\pi \mathrm{i} (h_1 x_i + h_2 y_i)}.\]
For a fixed vector $(h_1, h_2) \neq (0,0)$, we have
\[
e^{2\pi i(h_1x_i+h_2y_i)}
=
e^{2\pi i P_{n,\mathbf h}(i)},
\]
where
\[
P_{n, h_1,h_2}(X)
=
h_1g(r_n+X)
+h_2\bigl(g(s_n+X)-g(r_n+X)\bigr).
\]
In particular, $P_{n,h_1,h_2}$ is a polynomial in $X$. 

\textbf{Case 1}: $h_1 = 0$ and $h_2 \neq 0$. The polynomial $P_i$ drops to degree $\ell-1$ in $i$, with leading coefficient $\alpha = h_2 \ell b_\ell (s_n-r_n)$.

Let $\tau \geq 1$ be a real parameter to be chosen later. By Dirichlet's approximation theorem, we can find coprime integers $p, q$ with $q \leq \tau$ such that
\[\left|\alpha - \frac{p}{q}\right| \leq \frac{1}{q \tau} \leq \frac{1}{q^2}.\]
It follows that \[\left| b_\ell - \frac{p}{q h_2 \ell (s_n-r_n)} \right| \leq \frac{1}{q |h_2| \ell (s_n-r_n) \tau}.\]
Since $b_\ell$ is badly approximable, there exists an absolute constant $\kappa > 0$ such that for any rational ${P}/{Q}$, we have $|b_\ell - {P}/{Q}| \geq {\kappa}/{Q^2}$. Therefore \[\frac{\kappa}{q^2 |h_2|^2 \ell^2 (s_n-r_n)^2} \leq \frac{1}{q |h_2| \ell (s_n-r_n) \tau}.\]
We deduce that \[q \geq \frac{\kappa \tau}{|h_2| \ell (s_n-r_n)}.\]
   
   When $\ell>2$, we apply Weyl's quantitative bound for exponential sums of polynomials of degree $\ell-1$ (see e.g. \cite[Lemma~2.4]{Vaughan-1997}), for $\varepsilon>0$ we have
   \begin{equation}\label{eq: Weyl quantitative bound}
       |S_n(0, h_2)| \ll_{\ell,\varepsilon} (t_n-s_n)^{\varepsilon}\left( \frac{1}{t_n-s_n} + \frac{1}{q} + \frac{q}{(t_n-s_n)^{\ell-1}} \right)^{2^{2-\ell}}
   \end{equation}
   where the implied constant depends on $\ell$ and $\varepsilon$. We need to bound the right-hand side of \eqref{eq: Weyl quantitative bound}. Since $q \leq \tau$, we have \[\frac{1}{q} + \frac{q}{(t_n-s_n)^{\ell-1}} \leq \frac{|h_2| \ell (s_n-r_n)}{\kappa \tau} + \frac{\tau}{(t_n-s_n)^{\ell-1}}.\]
   To minimize the right-hand side, we choose \[\tau = \sqrt{\frac{|h_2| \ell}{\kappa} (s_n-r_n)(t_n-s_n)^{\ell-1}}.\] We obtain \[\frac{1}{q} + \frac{q}{(t_n-s_n)^{\ell-1}} \leq 2 \sqrt{ \frac{|h_2| \ell}{\kappa} \frac{s_n-r_n}{(t_n-s_n)^{\ell-1}} }.\]
 Since   $t_n-s_n \geq (\rho - 1)(s_n-r_n-1)$, we obtain $$|S_n(0,h_2)|\ll_{\ell,\varepsilon}(t_n-s_n)^\varepsilon\left(\frac{1}{t_n-s_n}+\frac{1}{(t_n-s_n)^{(\ell-2)/2}}\right)^{2^{2-\ell}}.$$
Therefore, when $\ell>2$, with $a:=\min\{1,\frac{\ell-2}{2}\}$ and $\varepsilon=a2^{1-\ell}$, from \eqref{eq: Weyl quantitative bound} we have
$$|S_n(0,h_2)|\ll_\ell (t_n-s_n)^{\varepsilon-a2^{2-\ell}}=(t_n-s_n)^{-a2^{1-\ell}}\ll (\rho-1)^{-a2^{1-\ell}}\leq (\rho-1)^{-2^{-\ell}}$$
for $n$ sufficiently large.

When $\ell=2$, we apply Weyl's method \cite[Section 8.2]{Iwaniec-Kowalski-2004} to deduce that 
\[
|S_n(0,h_2)|
\le
\frac{1}{2(t_n-s_n)\,
\|2h_2b_\ell(s_n-r_n)\|}.
\]
We estimate as previously to obtain, for all sufficiently large $n$, that
\[
|S_n(0,h_2)|
\le
\frac{2|h_2|}{\kappa(\rho-1)}.
\]

Thus, in any case of $\ell\geq2$, there is some $C_1=C_1(\ell,H,b_\ell)>0$ independent of $\rho$ such that
\begin{equation}\label{eq: Case 1}|S_n(0, h_2)| \leq C_1(\rho-1)^{-2^{-\ell}} \end{equation}
for $n$ sufficiently large.


\textbf{Case 2}: $h_1 \neq 0$. The polynomial $P_i$ has degree $\ell$ in $i$, with irrational leading coefficient $h_1 b_\ell$. We then argue by using Weyl's quantitative bound \cite[Lemma~2.4]{Vaughan-1997} as in the first case to deduce that there exists a constant $C_2=C_2(\ell,H,b_\ell)>0$ independent of $\rho$ such that
\begin{equation}\label{eq: Case 2}
    |S_n(h_1,h_2)|\leq C_2(\rho-1)^{-2^{\ell}}
\end{equation}
for all $n$ sufficiently large.

Now we combine these two cases to bound $D_n$. It follows that if we fix $H$ sufficiently large such that \[\frac{9}{2(H+1)} < \frac{1}{3} \left( \frac{\Delta^2}{9} - \epsilon \right),\] choose $\rho$ sufficiently large (which can be done uniformly since $|h_1|,|h_2| \leq H$) such that the sum \[\sum_{h_2 \neq 0} \frac{1}{r(0,h_2)}|S_n(0, h_2)|\] associated with \eqref{eq: Case 1} of Case 1 is bounded above by $ {4}/{27} \left( {\Delta^2}/{9} - \epsilon \right)$, and so that the sum \[\sum_{h_1 \neq 0} \frac{1}{r(h_1,h_2)}|S_n(h_1, h_2)|\] associated with \eqref{eq: Case 2} of Case 2 is also bounded above by ${4}/{27} \left( {\Delta^2}/{9} - \epsilon \right)$, then for all sufficiently large $n$, we have
\begin{equation}\label{eq: upper bound for Dn}
    D_n < \frac{\Delta^2}{9} - \epsilon.
\end{equation}

By assumption, we have $c([\gamma_0, \gamma_1)) \neq c([\gamma_1, \gamma_2))$. We then define 
\[ \Omega = \left[ \gamma_1 - \frac{\Delta}{3}, \gamma_1 \right) \times \left[ \frac{\Delta}{3}, \frac{2\Delta}{3} \right] \subset [0, 1)^2 \] whose Lebesgue measure is $\mu(\Omega) = {\Delta^2}/{9}$. Note that if $\vec{v}_i^{(n)} \in \Omega$, then $x_i \in [\gamma_0, \gamma_1)$ and $x_i + y_i \in [\gamma_1, \gamma_2)$. This  forces $a_{i+r_n} \neq a_{i+s_n}$. It follows from the definition of the discrepancy $D_n$ that the number of mismatches between $\mathbf{a}[r_n+1,r_n+t_n-s_n]$ and $\mathbf{a}[s_n+1,t_n]$ is at least
\[(t_n-s_n)(\mu(\Omega) - D_n) =(t_n-s_n)\left( \frac{\Delta^2}{9} - D_n\right). \]
Therefore, the $(\epsilon, \delta)$-closeness implies that
\[ \epsilon \geq \frac{\Delta^2}{9} - D_n, \text{ i.e. }D_n \geq \frac{\Delta^2}{9} - \epsilon ,\]
which contradicts~\eqref{eq: upper bound for Dn}. Thus $\mathbf{Rdio}(\mathbf{a}) < \infty$ as desired. 
\end{proof}

\begin{theorem}\label{theorem: Dio bracket word}
    Let $g(x) = \sum_{s=0}^\ell b_s x^s$ with $\ell\geq2$. Let $\mathbf{a}$ be defined by $a_n = c(\{g(n)\})$, where $c\colon[0,1)\to\Sigma$ is a non-constant piecewise-constant function with partition boundaries $0=\gamma_0<\gamma_1<\cdots<\gamma_d=1$.  Assume that the coefficients $b_1, \dots, b_\ell$ satisfy the following Diophantine condition: there exist a sequence of integers $q_n \to \infty$ and a sequence of real numbers $N_n$ such that:
    \begin{itemize}
        \item $0<\max_{1 \leq s \leq \ell} \|q_n b_s\| \leq q_n^{-N_n}$,
        \item $\lim_{n \to \infty} {N_n \ln(q_n)}/{q_n} = \infty$.
    \end{itemize}
    Then $\mathbf{Dio}(\mathbf{a}) = \infty$.
\end{theorem}

\begin{proof}
We will follow the strategy used in the proof of Theorem~\ref{theorem: Dio coding rotation} (when $\theta$ is well approximable). 

Let $p_{s,n} \in \mathbb{Z}$ such that $b_s = {p_{s,n}}/{q_n} + \delta_{s,n}$ with $|\delta_{s,n}| \leq q_n^{-N_n-1}$.
We decompose $g(x) = g_{rat}(x) + g_{rem}(x)$ into a rational part and a remainder part:
\[ g_{rat}(x) = b_0 + \sum_{s=1}^\ell \frac{p_{s,n}}{q_n} x^s, \quad g_{rem}(x) = \sum_{s=1}^\ell \delta_{s,n} x^s. \]
Observe that $g_{rat}(x + q_n) = g_{rat}(x) \pmod 1$.

We set $L_n = q_n^K-1$, where $K = \lfloor {(N_n - 2)}/{\ell} \rfloor$. For any integer $i \in [0, L_n + q_n]$ and for $n$ sufficiently large, we have $i \leq 2q_n^K$. Thus,
\[ |g_{rem}(i)| \leq \sum_{s=1}^\ell |\delta_{s,n}| i^s \leq \sum_{s=1}^\ell q_n^{-N_n-1} (2q_n^K)^\ell. \]
Since $K\ell \leq N_n - 2$, we deduce that
\[ |g_{rem}(i)| \leq \ell 2^\ell q_n^{-N_n-1} q_n^{N_n-2} = \ell 2^\ell q_n^{-3}. \]

Once again, we set  
\[ \Delta = \min_{1 \leq j \leq d} (\gamma_j - \gamma_{j-1}) >0. \] 
Note that for any integer $i$, $g_{rat}(i) \pmod 1$ belongs to $b_0 + \mathbb{Z}/q_n \pmod 1$. The distance between any two distinct points in $b_0 + \mathbb{Z}/q_n \pmod 1$ is at least ${1}/{q_n}$. For an index $i \in [0, L_n]$ to be a mismatch (meaning $a_i \neq a_{i+q_n}$), at least one of $g(i)$ or $g(i+q_n)$ must fall close to a boundary $\gamma_j$. Since both $|g_{rem}(i)|$ and $|g_{rem}(i+q_n)|$ are bounded by $\ell 2^\ell q_n^{-3}$, it follows that $g_{rat}(i) \pmod 1$ must fall within an $\ell 2^\ell q_n^{-3}$-neighborhood of some $\gamma_j$. Because $2(\ell 2^\ell q_n^{-3}) < {1}/{q_n}$ for all sufficiently large $n$, each $\gamma_j$ can have at most one associated point $V_{j,n} \in b_0 + \mathbb{Z}/q_n \pmod 1$ in its $\ell 2^\ell q_n^{-3}$-neighborhood. Thus, if $i$ is a mismatch, then $g_{rat}(i) = V_{j,n} \pmod 1$ for some $1 \leq j \leq d$.

Now we count the number of mismatches $i \in [0, L_n]$. If an index $i$ is a mismatch, the values $g(i)$ and $g(i+q_n)$ must fall on opposite sides of some $\gamma_j$. Because $g_{rat}(i) = g_{rat}(i+q_n) = V_{j,n} \pmod 1$ for some $j$, it follows that $g_{rem}(i)$ and $g_{rem}(i+q_n)$ must fall on opposite sites of $\gamma_j - V_{j,n}$. Thus, by the intermediate value theorem, the polynomial $g_{rem}(x) - (\gamma_j - V_{j,n})$ must have a real root in $[i, i+q_n]$. Note that $\deg(g_{rem}(x) - (\gamma_j - V_{j,n})) \leq \ell$, so this polynomial has at most $\ell$ real roots. Since a real root can belong to the interval $[i, i+q_n]$ for at most $q_n + 1$ integers $i$, the total number of mismatches within $[0, L_n]$ is bounded above by $\ell d (q_n+1)$.

Next, we use the telescoping argument from the proof of Theorem~\ref{theorem: Dio coding rotation}.
Assume that the mismatches occur at indices $i_1 < \dots < i_r < L_n+1 = q_n^K$ for some $r\leq \ell d(q_n+1)$. We set $i_0=0$ and $i_{r+1}=q_n^K$, and define
\[ M_n = \max_{0 \leq j \leq r} \frac{i_{j+1}+q_n}{i_j + 1 + q_n}. \]
Using induction as in the proof of Theorem~\ref{theorem: Dio coding rotation}, we obtain
\begin{align*}
\log(2M_n)&\geq \dfrac{1}{r+1}\log\left(\dfrac{q_n^K+q_n+1}{q_n+1}\right)\\&\geq\dfrac{1}{r+1}\log\left(\frac{q_n^{K-1}}{2}\right)\\&=\dfrac{1}{r+1}((K-1)\log q_n-\log 2).     
\end{align*}
Since $r\leq \ell d(q_n+1)$ and $K - 1 = \left\lfloor \frac{N_n - 2}{\ell} \right\rfloor - 1 \geq \frac{N_n - 2\ell - 2}{\ell},$ we deduce that
$$\log(2M_n)\geq \dfrac{1}{\ell d(q_n+1)+1}\left(\frac{N_n - 2\ell - 2}{\ell}\log q_n-\log 2\right). $$
As $n \to \infty$, our assumption $\lim_{n \to \infty} {N_n \ln(q_n)}/{q_n} = \infty$ implies that $M_n \to \infty$. Therefore, $\mathbf{Dio}(\mathbf{a}) = \infty$ as desired.
\end{proof}

\begin{remark}
    Using the recursive relation of the continued fraction convergents, one can easily construct examples that satisfy the Diophantine assumption in Theorem~\ref{theorem: Dio bracket word}. 
\end{remark}

Using the same strategy, but restricting to the monomial $g(x) = \beta x^\ell$, we obtain the following result with a weaker Diophantine assumption on $\beta$.

\begin{theorem}\label{theorem: monomial bracket}
Let $\beta$ be an irrational number with continued fraction convergents $({p_n}/{q_n})$. Let $g(x) = \beta x^\ell$ and let $c\colon[0,1)\to\Sigma$ be a non-constant piecewise-constant function with \emph{rational} partition boundaries. Let $\mathbf{a}$ be defined by $a_n = c(\{g(n)\})$. If 
\[ \limsup_{n \to \infty} \frac{q_{n+1}}{q_n^\ell} = \infty, \] 
then $\mathbf{Dio}(\mathbf{a}) = \infty$.
\end{theorem}

\begin{proof}
For each $1\leq j\leq d$, we may write $\gamma_j = B_j/Q$ for some integers $B_j$ and $Q$. We write $\beta = {p_n}/{q_n} + \delta_n$, where \[|\delta_n| \leq \frac{1}{q_n q_{n+1}}.\]Then we decompose $g(x) = g_{rat}(x) + g_{rem}(x)$ where
\[ g_{rat}(x) = \frac{p_n}{q_n} x^\ell, \quad g_{rem}(x) = \delta_n x^\ell. \]
We set $L_n = \left\lfloor \left({q_{n+1}}/{2Q}\right)^{1/\ell} \right\rfloor - q_n$. For $i \in [0, L_n+q_n]$, we have
\[ |g_{rem}(i)| = |\delta_n| i^\ell  \leq \frac{1}{q_n q_{n+1}} \frac{q_{n+1}}{2Q} = \frac{1}{2 Q q_n}. \]

We show that all $i \in [1, L_n]$ satisfy $a_i = a_{i+q_n}$. Indeed, note that $g_{rat}(i)=g_{rat}(i+q_n) \pmod 1\in \mathbb{Z}/{q_n}$, so its distance to any $\gamma_j = B_j/Q$ is either $0$ or $\geq{1}/{Q q_n}$. Because $g_{rem}(i)$ and $g_{rem}(i+q_n)$ have the same sign and both have absolute value $<1/Qq_n$, $g(i)$ and $g(i+q_n)$ must fall into the same partition interval. It follows that  $a_i = a_{i+q_n}$.

Now, our assumption yields that the ratio 
\[ \frac{L_n}{q_n} > \frac{\left(\frac{q_{n+1}}{2Q}\right)^{1/\ell} - q_n - 1}{q_n} = \frac{1}{(2Q)^{1/\ell}} \left( \frac{q_{n+1}}{q_n^\ell} \right)^{1/\ell} - 1 - \frac{1}{q_n} \]
tends to infinity along a subsequence. Therefore $\mathbf{Dio}(\mathbf{a}) = \infty$ as desired.
\end{proof}

\begin{remark}
The Diophantine condition $\limsup_{n \to \infty} {q_{n+1}}/{q_n^\ell} = \infty$ is stronger than the well-approximable condition. 
\end{remark}

\section{Final remarks}\label{section: conjectures}

In this final section, motivated by Table~\ref{table:examples}, we propose several questions regarding the refined Diophantine exponent.

\begin{question}
    Over a binary alphabet, is $[1,\infty]$ the spectrum of $\mathbf{Rdio}$? 
\end{question}

In \cite{Peltomaki-2024}, Peltom\"{a}ki studied the Diophantine exponent of \emph{regular episturmian} words (first introduced in \cite{Droubay-Justin-Pirillo-2001}) and conjectured that, for an aperiodic episturmian word $\mathbf{a}$, $\mathbf{Dio}(\mathbf{a})<\infty$ if and only if $\mathbf{a}$ has bounded partial quotients; see Conjecture 8.4 in \emph{loc. cit.} For $\mathbf{Rdio}$, we expect that it is always infinite for this family of words.

\begin{question}
Let $\mathbf{a}$ be an aperiodic episturmian word. Is it always true that $\mathbf{Rdio}(\mathbf{a})=\infty$?
\end{question}

Finally, we remark on a result of Schaeffer and Shallit \cite{Schaeffer-Shallit-2012}, which states that the critical and Diophantine exponents of an automatic word are computable and are always either rational numbers or infinite. The same holds for the \emph{initial critical} exponent introduced in \cite{Berthe-Holton-Zamboni-2006}. Their proof makes use of first-order logic. However, for $\mathbf{Rdio}$, the difficulty lies in the appearance of $(\epsilon,\delta)$-closeness. We expect the following to be true.

\begin{question}
    Is it true that the refined Diophantine exponent of an automatic word is always either a rational number or infinite?
\end{question}

\subsection*{Acknowledgement}
I am deeply grateful to Boris Adamczewski for his helpful discussions and suggestions, and to Charles Favre for his constant encouragement. Part of this work was completed during my visit to the University of Waterloo; I thank the university for its hospitality. I would also like to thank Jason Bell, Colin Faverjon and Maria Clara Werneck for helpful conversations, Julien Melleray and Emmanuel Peyre for their questions, and Antoine Aurillard for several discussions on probability.

This project has received funding from the European Union’s MSCA-Horizon Europe, grant agreement No. 101126554. 

 \subsection*{Disclaimer}
Co-funded by the European Union. Views and opinions expressed are however those of the
author only and do not necessarily reflect those of the European Union. Neither the European Union nor the granting authority can be held responsible for them.

\noindent \includegraphics[height = 1cm]{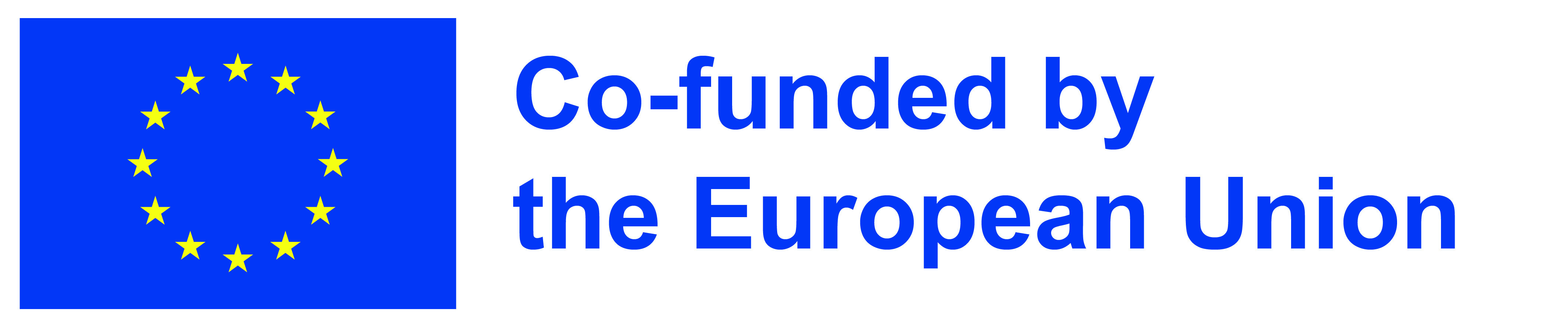}

{\scriptsize
\bibliographystyle{plain}
\bibliography{references}
}
\end{document}